\documentclass[11pt,oneside,english,reqno]{amsart}
\usepackage[T1]{fontenc}
\usepackage[utf8]{inputenc}
\usepackage[a4paper]{geometry}
\usepackage{mathrsfs}
\usepackage{amstext}
\usepackage{amsthm}
\usepackage{amssymb}
\usepackage{color}
\usepackage{hyperref}
\usepackage{enumerate}
\usepackage{verbatim} 
\usepackage{stmaryrd}
\usepackage{enumitem}
\usepackage{graphicx}
\makeatletter
\numberwithin{equation}{section}
\numberwithin{figure}{section}
\theoremstyle{plain}
\newtheorem{thm}{\protect\theoremname}
\theoremstyle{definition}
\newtheorem{defn}[thm]{\protect\definitionname}
\theoremstyle{remark}
\newtheorem{rem}[thm]{\protect\remarkname}
\theoremstyle{plain}
\newtheorem{lem}[thm]{\protect\lemmaname}
\theoremstyle{plain}
\newtheorem{prop}[thm]{\protect\propositionname}
\theoremstyle{plain}
\newtheorem{cor}[thm]{\protect\corollaryname}
\theoremstyle{plain}
\newtheorem{ass}{\protect\assumptionname}
\theoremstyle{plain}
\newtheorem{ex}[thm]{\protect\examplename}
\makeatother

\usepackage{babel}
\providecommand{\corollaryname}{Corollary}
\providecommand{\definitionname}{Definition}
\providecommand{\lemmaname}{Lemma}
\providecommand{\propositionname}{Proposition}
\providecommand{\remarkname}{Remark}
\providecommand{\theoremname}{Theorem}
\providecommand{\assumptionname}{Assumption}
\providecommand{\examplename}{Example}

\newcommand{\red}[1]{{\color{red} #1}}
\newcommand{\blue}[1]{{\color{blue} #1}}

\newcommand{\cB}{\mathcal{B}}

\newcommand{\cF}{\mathcal{F}}

\newcommand{\cI}{\mathcal{I}}

\newcommand{\cL}{\mathcal{L}}

\newcommand{\cP}{\mathcal{P}}

\newcommand{\EE}{\mathbb{E}}
\newcommand{\FF}{\mathbb{F}}

\newcommand{\PP}{\mathbb{P}}

\newcommand{\RR}{\mathbb{R}}

\newcommand{\TT}{\mathbb{T}}

\newcommand{\dd}{\mathop{}\!\mathrm{d}}

\newcommand{\norm}[1]{\left\lVert #1\right\rVert}

\begin{document}

\title[]{Pathwise regularization by noise for semilinear SPDEs driven by a multiplicative cylindrical Brownian motion}
\author{Florian Bechtold \and Fabian A. Harang }
\date{\today}

\address{Florian Bechtold: Fakult\"at f\"ur Mathematik,
	Universit\"at Bielefeld, 33501 Bielefeld, Germany}
\email{fbechtold@math.uni-bielefeld.de}

\address{Fabian A. Harang: Department of Economics, BI Norwegian Business School, Handelshøyskolen BI, 0442, Oslo, Norway.}
\email{fabian.a.harang@bi.no}

\thanks{ 
\\
\emph{Acknowledgment:} FB acknowledges support through the Bielefeld Young Researchers Fund. FH acknowledges support from the "Signatures for Images" project at the Center of Advanced studies (CAS) in Oslo, during which parts of this article was written. }


\begin{abstract}
 We prove a regularization by noise phenomenon for semilinear SPDEs driven by multiplicative cylindrical Brownian motion and singular diffusion coefficient, addressing an open problem in \cite{catellier2021pathwise}. The analysis is based on a combination of infinite dimensional generalizations of arguments in \cite{bechtold} as well as careful maximal regularity analysis for semilinear SPDEs and Volterra-sewing techniques developed in \cite{Harang_Tindel21}.
\end{abstract}

\maketitle

\section{Introduction}

We investigate existence of solutions to parabolic SPDEs of the form 
\begin{equation}\label{eq:first}
    \dd v_t(x) = \Delta v_t(x)\dd t + \sigma (v_t(x))\dd W_t(x) +\dd w_t, \quad (t,x)\in [0,T]\times \TT 
\end{equation}
where $\sigma$ is singular but integrable function, $W:\Omega\times [0,T]\times \TT \rightarrow \RR$ is a space-time stochastic process on a filtered probability space $(\Omega,\cF,\PP)$  which is white in time, and $w:[0,T]\rightarrow \RR$ is a continuous sample path of a stochastic process. Such type of SPDEs are often called stochastic heat equation with multiplicative space-time noise.  In the spirit of "pathwise regularization by noise" (see e.g. \cite{CATELLIER20162323,galeati2020noiseless,harang2020cinfinity}), we will in this article investigate the interplay between conditions on  the additive continuous path $w$ and on the nonlinear coefficient $\sigma$, under which \eqref{eq:first} is well posed.  
\\

In recent years much improvement has been made in showing how the addition of stochastic processes to otherwise ill-posed (i.e. non-existing and/or non-unique equations) ODEs makes the equations  well posed. While this fact has been well known since the early discoveries of Zvonkin \cite{Zvonkin_1974}, showing that the addition of a Brownian motion establishes (probabilistic) uniqueness for certain non-unique ODEs,  the research in this direction over the last ten years has been blooming, with the availability of several new tools and techniques that go beyond the Markovian setting. In particular, in \cite{CATELLIER20162323} Catellier and Gubinelli studied pathwise well-posedness of SDEs of the form 
\begin{equation}\label{eq:SDE fbm}
   \dd x_t = b(x_t)\dd t + \dd w_t 
\end{equation}
where the stochastic process $w$ is a fractional Brownian motion with Hurst parameter $H\in (0,1)$. The authors showed that this equation has a pathwise unique solution for any $b\in \cB^\alpha_{\infty,\infty}$ for $\alpha>1-\frac{1}{2H}$ where $\cB^\alpha_{p,q}$ denotes the Besov space of regularity $\alpha$ and integrability indices $p$ and $q$. 
One of the novel insights from that article is the importance of the regularity of the {\em averaged field}, defined as the integral function 
\[
A_t(x)=\int_0^t b(x+w_r)\dd r,
\]
and how this plays a central role for proving existence and uniqueness of equations of the form \eqref{eq:SDE fbm}. 
Another inspirational insight from this article is that it is the irregularity, or roughness, of the trajectories of $w$ that seems to provide a regularizing effect for the SDE. 
This interpretation of the phenomena has later been highlighted in more detail through the connection with occupation measures and local times (see \cite{harang2020cinfinity,galeati2023prevalence}), through the local time formula 
\begin{equation}\label{averaged field}
    \int_0^t b(x-w_r)\dd r= b\ast L_t(x),
\end{equation}
where $L$ denotes the local time associated to the path $w$, see Section \ref{averaging operators} for more details. 
These techniques have later been extended upon in different directions, improved, and discussed in a great number of articles, see e.g. \cite{harang2020cinfinity,galeati2020noiseless,bechtold2,bechtold2023pathwise,Le20,athreya2022wellposedness,catellier2021pathwise,galeatiharang}. 
\\

One direction of the {\em regularization by noise} program that has received less attention is the problem of regularization when the SDE has multiplicative noise. That is, one considers a classical SDE controlled by a stochastic process $\beta$ of the form 
\begin{equation}\label{eq:second SDE}
    \dd x_t = \sigma(x_t)\dd \beta_t + \dd w_t, 
\end{equation}
 and $w$ plays the role of a potentially regularizing path. Again, in this problem we are looking for a class of paths $w$ with the property that that it restorers existence and uniqueness to \eqref{eq:second SDE} for non-Lipschitz $\sigma$, in cases when it is known that the equation is ill-posed if $w\equiv 0$. 
To the best of our knowledge, this problem was initially studied by Galeati and one of the authors of the current article in \cite{galeatiharang}, where the case when $\beta$ was a fractional Brownian motion with $H>\frac{1}{2}$ was studied. It is there proved that the regularizing effect that the noise $w$ may have on the equation also depends on the roughness of the stochastic process $\beta$. The closer $H$ was to $\frac{1}{2}$, the better regularizing properties of $w$ is required in order to obtain existence and uniqueness of \eqref{eq:second SDE}. The analysis is also there strongly influenced by the concept of averaged fields, as defined in \eqref{averaged field}, and inspired by the techniques of pathwise regularization by noise. However, the equations them self where considered as truly stochastic, and thus techniques mixing pathwise and probabilistic considerations where used to obtain stochastic strong solutions. 
Since the analysis there relied upon the condition $H>\frac{1}{2}$, the case when $\beta$ is a multiplicative Brownian motion was excluded, and thus regular It\^o SDEs could not be considered. However, in a more recent article \cite{bechtold}, Hofmanov\'a and the other author of the current  article has managed to prove (stochastic) weak existence of solutions to \eqref{eq:second SDE} in the case when $\beta$ is a standard Brownian motion, and $w$ is a regularizing path. However, this also comes at the price of being able to treat only singularities of the form $\sigma\in L^p$, as opposed to distributional $\sigma$ in \cite{galeatiharang}. The approach taken there consists again of a combination of pathwise techniques based on averaged fields as in \eqref{averaged field} and local times and classical theory of weak solutions to stochastic equations. In particular, deriving an a-priori bound on H\"older scale based on smooth approximations in combination with certain regularity estimates of averaged fields allows to extract a convergent sub-sequence after which an identification of the limit is done. 
\\

Regularization by noise for stochastic equations with multiplicative noise has also been studied in \cite{catellier2021pathwise}. There, stochastic partial differential equations (SPDEs) with multiplicative spatial noise where considered (such as the Parabolic Anderson Model), and regularization by noise where proved for equations of the form 
\begin{equation*}
        \dd v(x) = \Delta v_t(x)\dd t +\sigma(v_t(x))\xi(x)\dd t +\dd w_t,\quad (t,x)\in \RR_+\times \TT. 
\end{equation*}
Here $\sigma$ is a nonlinear function, $\xi$ is a stochastic spatial white noise  and $w$ is a continuous path (only depending on time), providing the regularizing effect. Pathwise existence and uniqueness were obtained there, even for highly singular $\sigma$, but the techniques could not be generalized to allow for a  time dependence in the multiplicative noise $\xi$.  
\\

In the current article the goal is to prove a regularization by noise phenomena for SPDEs with multiplicative {\em space time} noise. In particular, considering SPDEs of the form 
\begin{equation} \label{eq:SPDE space time noise}
    \dd v(x) = \Delta v_t(x)\dd t +\sigma(v_t(x))\xi_t(x)\dd t +\dd w_t,\quad (t,x)\in \RR_+\times \TT. 
\end{equation}
Here $\sigma$ is again a nonlinear function, $\xi$ is a stochastic space-time noise (to be specified) and $w$ is again a real valued continuous path (only depending on time). In spirit of the regularization by noise program, we will combine the techniques of 
\cite{bechtold} and \cite{catellier2021pathwise} in order to prove stochastic weak existence of this equation when $\sigma$ is a potentially singular and the multiplicative noise is fully depending on space and time, providing an answer to one open problem left in the conclusion of \cite[Sec. 6]{catellier2021pathwise}. 
\\

Our analysis is based on arguments related to infinite dimensional stochastic equations. To this end, we write the SPDE in \eqref{eq:SPDE space time noise} in its mild form (see e.g.\cite{da_prato_zabczyk_1992}), to obtain 
\begin{equation*}
    v_t = P_t v_0 + \int_{0}^t P_{t-s} \sigma (v_s) \xi_s \dd s +w_t, \quad t\in [0,T]. 
\end{equation*}
where $\{P_{t}\}_{t\in [0,T]}$ is the semi-group generated by the heat equation. 
Using the translation trick $u = v+w$, then formally $v$ solves 
\begin{equation}\label{eq:integral eq}
    u_t = P_tv_0 + \int_0^t P_{t-s} \sigma(u_s+w_s)\xi_s \dd s. 
 \end{equation}
Due to the dependence of $P$ on the difference $t-s$, we may view this equation as an infinite dimensional Volterra equation.
In our analysis, we will see the space time noise $\int_0^t \xi_s \dd s\equiv  W_t$ where $W$ is a cylindrical Wiener process on a separable Hilbert space.  
We begin to show how the integral appearing in Equation \eqref{eq:integral eq} can truly be seen as an infinite dimensional It\^o integral in Section \ref{integral def}. We then continue in Section \ref{sec:Volterra sewing} with a discussion of the regularity of averaged fields in the presence of a Volterra integral kernel. In Section \ref{sec:apriori and tightness} we prove two different a-priori bounds of the integral equation in  \eqref{eq:integral eq}, which is used in combination with the Aubin-Lions lemma to prove tightness. After obtaining potential solution candidates, we provide in Section \ref{sec:identification} an identification of the limiting equation based on martingale techniques, concluding our results.  

 \subsection{Notation} For $p\geq 1$ we let $L^p_x=L^p(\FF)$ denote the regular Lebesgue spaces, where $\FF$ is either $\RR $ or $\TT$ the one dimensional torus. The exact choice will be evident from the computations. Similarly, we let $L^p_\omega := L^p(\Omega) $.  For $\alpha\in \RR$ let $H^\alpha$ denote the $L^2_x$ Bessel-potential space (also known as the fractional Sobolev space). For $\alpha \in (0,1)$ We denote by $C^\alpha_t:=C^\alpha([0,T])$ the classical space of $\alpha$-H\"older continuous functions over some time interval $[0,T]$ that will be assumed to be fixed throughout the article. 
To shorten notation of the spaces $L^p(\RR;E)$ or $C^\alpha([0,T];E)$ for some Banach space we write $L^p_xE$ and $C^\alpha_t E$. For two Banach spaces $U$ and $E$ we let $\cL(U,E)$ denote the space of bounded linear operators from $U$ to $E$. Whenever $U=E$ we simply write $\cL(U)$. Similarly when  we let $\cL_2(L^2_x)$ denote the set of Hilbert Schmidt operators $A$, which we equip with the norm 
\begin{equation}
    \| A\|_{\cL_2(L^2_x)}^2=\sum_{i=1}^{\infty} \|A e_i\|_{L^2_x}^2,  
\end{equation}
where $\{e_i\}_i$ is an orthonormal basis in $L^2_x$.

\subsection{Main results}\label{main idea}
Before presenting the main findings, we will briefly discuss the infinite dimensional setting of the Brownian noise we consider here, and the standing assumptions on the structure of the singular coefficients $ \sigma$. 

Let $\Delta$ be the Laplacian endowed with periodic boundary conditions on the one dimensional torus $\mathbb{T}$ . 
Let $W$ be a cylindrical Wiener process on a separable Hilbert space $L^2_x :=L^2(\mathbb{T})$, i.e. formally 
 \[
 W_t=\sum_{k=1}^\infty e_k \beta_t^k
 \]
 where $(e_k)_k$ is an orthonormal basis in $L^2_x$ 
 and $(\beta^k)_k$ a sequence of independent Brownian motions.
 We assume $\sigma: L^2_x\to \cL(L^2_x)$ to be of the form
 \begin{equation}
      (\sigma (u))(\cdot)=\sum_k \sigma_k(u)\langle e_k, \cdot \rangle, \quad u\in L^2_x
      \label{form of sigma}
 \end{equation}
 for some measurable functions $\sigma_k:\RR\to \RR$. We will denote throughout the remainder of the paper
 \[
 \Sigma^2(x):=\sum_k \sigma_k^2(x), \quad x\in \RR.
 \]
Our main assumption later on will be $\Sigma^2\in L^p_x(\RR) $ for some $p\geq 1$, 
and as an example we consider the case where locally
 \[
|\Sigma^2(x)|\simeq \frac{1}{|x|^\gamma}
 \]
 for some $\gamma<1$. Remark in particular that in this setting, we have for some function $u\in L^2_x$ that the Hilbert-Schmidt norm satisfies
 \begin{equation}
 \label{control of operator norm}
     \begin{split}
         \norm{\sigma(u)}_{\cL_2(L^2_x)}^2&:=\sum_k \norm{\sigma_k(u)}_{L^2_x}^2=\sum_k \int_\mathbb{T} \sigma_k^2(u)\dd x=\int_{\mathbb{T}}\Sigma^2(u)\dd x
     \end{split}
 \end{equation}

We now provide  a definition of the solution concept we use here. This is indeed based on classical stochastic weak solutions, but adjusted to the "pathwise regularization" program that will be used here. 
It is worth noting here that the perturbation noise $w$ that enters the equation \eqref{eq:integral eq} is one-dimensional. This is both to keep the analysis clear but also to ensure that the noise is simply compatible with our later infinite dimensional analysis of the stochastic equations.

\begin{defn}
Let $w^H$ be a $1$-dimensional fractional Brownian motion of Hurst parameter $H$ on $(\Omega^H, \mathcal{F}^H, \mathbb{P}^H)$. Denote by $B^H$ the set of full probability on which the path $w^H(\omega^H)$ admits a local time $L(\omega^H)$. Let $\sigma$ be defined as in \eqref{form of sigma}. We say that $(\Omega, \mathcal{F}, (\mathcal{F}_t)_t, (W_t)_t, (u_t)_t)$ is a weak solution to 
\begin{equation}
    du_t=\Delta u_tdt+\sigma(X_t-w^H_t)dW_t, \qquad u_0\in L^2_x
    \label{main problem}
\end{equation}
if $(W_t)_t$ is  a $(\mathcal{F}_t)_t$ adapted cylindrical Brownian motion on $L^2(\mathbb{T})$, $(u_t)_t$ is a $(\mathcal{F}_t)_t$ progressively measurable process and if there exists a measurable set $S^H\subset B^H$ with $\mathbb{P}^H(S^H)=1$ such that for all $\omega^H\in S^H$, the germ $A_{s,t}=\int_\mathbb{T}(\Sigma^2*L_{s,t}(\omega^H))(u_s)dx$ admits a sewing in $L^1(\Omega)$, 
\[
\mathbb{E}\left[ \int_0^T \norm{\sigma}^2_{\cL_2(L^2(\mathbb{T}))}(u_s-w^H_s(\omega^H))ds\right]=\mathbb{E}\left[ \int_0^T \int_\mathbb{T}\Sigma^2(u_s-w^H_s(\omega^H))dxds\right]:=\norm{(\mathcal{I}A_T)}_{L^1(\Omega)}<\infty.
\]
Furthermore, for all $\omega^H\in S^H$, $\mathbb{P}$-almost surely, the solution satisfies
\[
u_t=u_0 + \int_0^t\Delta u_sds+\int_0^t\sigma(u_s-w^H_s)dW_s, 
\]
for any $t\in [0, T]$, where the stochastic integral is understood in the sense of Lemma \ref{extending the integral}.  
\label{notion of solution}
\end{defn}

Throughout the remainder of the paper,  we will make the following main assumption on the relation between the parameters $p$ and $H$. 
\begin{ass}
We assume $p\geq 1$, $\gamma_0\in (1/2, 1)$ and $H\in (0,1)$ satisfy the following conditions
\begin{equation}
\label{condition in H}
    2H<\left(1+\frac{1}{(p\wedge 4/3)}\right)^{-1}, \qquad \gamma_0<1-\left(4+\frac{1}{p/4\wedge 1/3} \right)^{-1}.
\end{equation}

\label{fixing null set}
\end{ass}

We will now present the main theorem to be proven in this article.

\begin{thm}\label{main theorem}
    Suppose $H\in (0,1)$, $\gamma_0\in (1/2, 1)$ and $p\geq 1$ satisfy Assumption \ref{fixing null set}, and suppose $u_0\in H^{\gamma_0}$. Then for any $\sigma$ of the form \eqref{form of sigma} such that $\Sigma^2\in L^p_x(\RR)$ the problem \eqref{main problem} admits a weak solution $u$ in the sense of Definition \ref{notion of solution}. Moreover, for some $m\geq 2$, we have $\mathbb{P}^H$-almost surely
    \[
    u\in L^m(\Omega; C^{\gamma_0}([0, T]; L^2_x))\cap L^m(\Omega,  L^\infty([0, T]; H^{\gamma_0}))
    \]

\end{thm}

As a particular application and immediate consequence, we illustrate Theorem \ref{main theorem}  with a corollary in the case when $w$ is a fractional Brownian motion and $\sigma$ is singular in the sense that  $|\Sigma(x)|\simeq \frac{1}{|x|^\gamma}$ for some $0<\gamma<1$: 
\begin{cor}
For $p\geq 1$ and $H\in (0,1)$ suppose Assumption \ref{fixing null set} holds. Given the setting of Theorem \ref{main theorem}, suppose $\sigma$ is a singular function, such that $|\Sigma^2(x)|\simeq \frac{r(x)}{|x|^\gamma}$, where $r(x)\equiv 1$ on some subset $A\subset \RR$ of finite size, and decays exponentially outside this set. Then if $\gamma<\frac{1}{p}$, then there exists a weak solution to \eqref{eq:SPDE space time noise} for $\PP$-almost all samples paths of the Brownian motion $w$ with Hurst parameter $H$.     
\end{cor}

 \begin{rem}
     Our two results above are stated under the assumption that the continuous path $w^H$ is sampled from a fractional Brownian motion. However, this assumption can easily be generalized to any continuous path with a sufficiently regular local time, see in particular Section \ref{averaging operators} and the results therein. As the fractional Brownian motion is a standard representation of such regularizing paths and a familiar Gaussian process to work with, we choose to do all our analysis here with this process.  
 \end{rem}
 
 \section{Pathwise regularization by noise}\label{averaging operators}
 We provide a brief overview of the concept of averaging operators as introduced in \cite{CATELLIER20162323}, which will be a central object in the remainder of the article. This section is to be accompanied with the concept of local times, which is briefly presented in the appendix \ref{app: local time section}.

While the study of potential regularizing effects of perturbations by stochastic processes has received much attention in the past decade (see e.g. \cite{CATELLIER20162323,galeati2023prevalence,harang2020cinfinity} and several references mentioned in the introduction), we will  for the reader's convenience begin by citing a result on the regularity of averaging operators associated with fractional Brownian motion. This process with its regularizing properties will serve as our main example throughout the remainder of the article, and will be denoted by $w$ throughout.

A particularly interesting relation is that $\mathbb{P}$-almost any realization of the fractional Brownian motion $\{w_t\}_t$ admits a local time $L$ (since $H<1$, and the noise is one dimensional) and we have the relation 
\[
T^{-w}_tf: (t, x)\mapsto \int_0^t f(x-w_s)\dd s=(f*L_t)(x).
\]
For more information about local times and their relation to the averaged field above, see e.g. \cite{harang2020cinfinity}.

The following result represents a slight adaptation of \cite[Theorem 3.1]{harang2020cinfinity} to our purposes taken from \cite[Theorem 3.1]{bechtold}.

\begin{thm}[Regularity of averaging operators]
Let $w$ be $1$-dimensional fractional Brownian motion of Hurst parameter $H\in (0,1)$ on $(\Omega, \mathcal{F}, \mathbb{P})$  and let $p \in [1, \infty)$. Then there exists a $\Omega$-nullset $\mathcal{N}$ such  that for any $\omega\in\mathcal{N}^c$, $w(\omega)$ admits a local time $L$ and for any $(\lambda, \gamma)$ satisfying
\begin{equation}
    \lambda<1/(2H)-1/(p\wedge 2), \qquad \gamma<1-(\lambda+1/2)H
    \label{local time parameters}
\end{equation}
we have $T^{-w}f\in C^\gamma_t C^{\lambda}_x$  provided $f\in L^p_x(\RR)$. Moreover we have for $f_1, f_2\in L^p_x(\RR)$ the stability property
\begin{equation}
    \norm{T^{-w}(f_1-f_2)}_{C^\gamma_t C^{\lambda}_x}\lesssim \norm{f_1-f_2}_{L^p_x(\RR)}.
    \label{stability}
\end{equation}
\label{regularity of averaging operator}
\end{thm}

%\red{see marked red below; The above assumption does not seem to guarantee $H<1/d$ to me... I.e. the pair $(H,p)=(1/3,2)$ satisfies the assumption, but if $d=100$, then $H$ is not smaller  than $1/100$...  I think assumption 1 need a relation with $d$ to make sure that $H<1/d$.}

In the following example we illustrate how this assumption can be used in combination with Theorem \ref{fixing null set} to derive useful regularity bounds that will be applied later.

\begin{ex}
    Let $\Sigma^2=\sum_k\sigma_k^2 \in L^p_x$ and let $w$ be a $(\Omega, \mathcal{F}, \mathbb{P})$ one-dimensional fractional Brownian motion of Hurst parameter $H$, where $p$ and $H$ satisfies Assumption \ref{fixing null set} 
 and there exists a local time for $\Omega$-almost any realization of $w$. By Theorem \ref{regularity of averaging operator}, we can find a $\Omega$-null set $\mathcal{N}$ independent of $\Sigma^2$ such that for any $\omega\in \mathcal{N}^c$ we have for some $\gamma_0, \gamma_1\in (1/2, 1)$ the inequalities
 \begin{equation}
     \norm{(T^{-w}_{s,t}\Sigma^2)}_{C^0_x}\lesssim \norm{\Sigma^2}_{L^p_x} |t-s|^{\gamma_0}, \qquad \norm{(T^{-w}_{s,t}\Sigma^2)}_{C^{1}_x}\lesssim \norm{\Sigma}_{L^p_x}  |t-s|^{\gamma_1}
     \label{regularity available}
 \end{equation}
 hold for any $\Sigma^2\in L^p_x$. Moreover, $\gamma_0, \gamma_1$ satisfy $\gamma_0/2+\gamma_1>1$. Indeed, remark that the condition $2H< (1+\frac{1}{p\wedge 2})^{-1}$ ensures $T^{-w}_{s,t}\Sigma^2\in C^{1}_x$ whereas $2H< (1+\frac{1}{4/3})^{-1}$ ensures that $\gamma_0, \gamma_1$ may be chosen such that $\gamma_0/2+\gamma_1>1$. The maximal $\gamma_0$ we may choose that satisfies these two conditions is 
 \begin{equation}
\gamma_0<1-\left(4+\frac{1}{p/4\wedge 1/3} \right)^{-1}.
     \label{maximal gamma}
 \end{equation}
\end{ex}

\begin{rem}
    In the remainder of the article we will treat $w$ as a fixed realization of a fractional Brownian motion, for which we have an associated local time that is well behaved (according to the above Theorem and assumption). We will therefore not make any probabilistic considerations with respect to the trajectory $w$, and keep all probabilistic estimates related to the infinite dimensional noise $W$, presented in more detail ion the next section. 
\end{rem}

\section{It\^o integrals in the presence of regularizing paths}
\label{integral def}
The following section is essentially an adaption of \cite[Section 3]{bechtold} to our infinite-dimensional setting. 
Before addressing the proof of Theorem \ref{main theorem}, let us first remark that even in the case of a singular diffusion coefficient $\sigma$, with $\Sigma^2\in L^p_x$, it is a-priori unclear why for progressively measurable $u\in C_tL^2_x$ the infinite dimensional stochastic integral 
\[
\int_0^t \sigma(u_r-w_r)\dd W_r
\]
appearing in Definition \ref{notion of solution} should even be a well-defined object. Indeed, as $\sigma$ is neither bounded nor of linear growth, it is at first sight unclear why in the below It\^o type isometry the right hand side should be finite; 
\begin{align*}
    \mathbb{E}\left[ \norm{\int_0^T \sigma(u_r-w_r)\dd W_r}_{L^2_x}^2\right]
    &=\mathbb{E}\left[ \left(\int_0^T\norm{\sigma(u_r-w_r)}_{\cL_2(L^2_x)}^2\dd r\right)\right]\\
    &=\mathbb{E}\left[ \left(\int_0^T\int_{\mathbb{T}}\Sigma^2(u_s-w_s) \dd x \dd r\right)\right]. 
\end{align*}
We remark again here that $w$ is now seen as a realization of a fractional Brownian motion, and thus a deterministic path.  
Already at this stage, we will therefore need to harness the regularizing properties of averaging operators as outlined in Subsection \ref{main idea}. More precisely, let $\sigma_\epsilon$ denote a cut-off mollification of $\sigma$,  by which we mean
\begin{equation}
    (\sigma_\epsilon(x))(\cdot)=\sum_k\sigma_{k, \epsilon}(x)\langle e_k, \cdot\rangle,
    \label{cut-off mollification}
\end{equation}
where $\sigma_{k, \epsilon}=(\sigma_k*\rho^\epsilon)\varphi_\epsilon$ and $\rho^\epsilon$ is a sequence of mollifiers and $\varphi_\epsilon$ is a smooth positive cut-off function. Let us denote $\Sigma^2_\epsilon(x):=\sum_k \sigma_{k, \epsilon}^2(x)$. It can be easily verified that for $\Sigma^2\in L^p$ and each $\epsilon>0$, there exists two constants  constant $c_\epsilon>0$ and $C_\epsilon>0$ such that  $\Sigma_\epsilon^2(x)\lesssim c_\epsilon^2$ for a.a. $x\in \RR$, and 
\begin{equation}
\norm{\sigma_\epsilon(u)}^2_{\cL_2(L^2_x)}\leq c_\epsilon^2, \qquad \norm{\sigma_\epsilon(u)-\sigma_\epsilon(v)}_{\cL_2(L^2_x)}\leq C_\epsilon \norm{u-v}_{L^2_x}.
    \label{mollified diffusion}
\end{equation}
For any $\epsilon>0$, we therefore have
\begin{align*}
    \mathbb{E}\left[ \norm{\int_0^T \sigma_\epsilon(u_r-w_r)\dd W_r}_{L^2_x}^2\right]&=\mathbb{E}\left[ \left(\int_0^T\int_{\mathbb{T}}\Sigma_\epsilon^2(u_s-w_s) \dd x \dd r\right)\right]\lesssim c_\epsilon^2 T,
\end{align*}
meaning that for any progressively measurable $u\in C_tL^2_x$, the stochastic integral is well defined. 
In Lemma \ref{identification} below, we first show that under certain regularity assumptions on $u$, we have the identification
\[
\mathbb{E}\left[ \left(\int_0^T\int_{\mathbb{T}}\Sigma_\epsilon^2(u_s-w_s) \dd x \dd r\right)\right]=\EE[(\mathcal{I}A^\epsilon)_{0,T}],
\]
where $A^\epsilon_{0,T}=\int_{\mathbb{T}}(\Sigma_\epsilon^2*L_{0,T})(u_s)\dd x$, and $\cI A^\epsilon$ denotes the sewing of $A^\epsilon$, and we recall that $L$ denotes the local time associated to $w$.  In a second step (Lemma \ref{extending the integral}), we exploit the gain of regularity due to the local time to show that the above sewing is stable in the limit $\epsilon\to 0$, i.e. $(\mathcal{I}A^\epsilon)\to (\mathcal{I}A)$, where $A=\int_{\mathbb{T}}(\Sigma^2*L_{s,t})(u_s)\dd x$. This allows to conclude that the sequence
\[
\left(
\int_0^t \sigma_\epsilon(u_r-w_r)\dd W_r\right)_\epsilon,
\]
of stochastic integrals is Cauchy, and thus to deduce the existence of a limit we will denote  by $\int_0^t \sigma(u_r-w_r)\dd W_r$. Note moreover that thanks to the approximation procedure employed in the construction of $\int_0^t\sigma(u_r-w_r)\dd W_r$, properties such as adaptedness naturally carry over.

\begin{lem}[Identification]
\label{identification}
Let $w$ be a fractional Brownian motion. Let $\omega\in \Omega$ such that $w(\omega)$ is locally $\alpha$-H\" older continuous for $\alpha<H$. Let $\sigma$ be of the form of \eqref{form of sigma}  and $\sigma_\epsilon$ a corresponding cut-off mollification defined in \eqref{cut-off mollification}. Let $m\geq 2$ and $u$ be a stochastic process satisfying 
\[
\norm{u}_{C^{\gamma_0/2}_tL^m_\omega L^2_x}^m=\sup_{s\neq t\in [0,T]}\frac{\mathbb{E}[\norm{u_{t}-u_s}_{L^2_x}^m]}{|t-s|^{m\gamma_0/2}}<\infty,
\]
for some $\gamma_0>0$. Then the germ
\[
A_{s,t}^\epsilon=\int_s^t \norm{\sigma_\epsilon(u_s-w_r)}_{\cL_2(L^2_x)}^2\dd r=\int_s^t\int_\mathbb{T} \Sigma_\epsilon^2(u_s-w_r)\dd x\dd r=\int_\mathbb{T} (\Sigma_\epsilon^2*L_{s,t})(u_s)\dd x,
\]
admits a sewing $\mathcal{I}A^\epsilon$ in $L^{m/2}(\Omega)$ and we have for any $t\in [0,T]$
\[
 \norm{(\mathcal{I}A^\epsilon)_t-\left(\int_0^t \int_\mathbb{T}\Sigma_\epsilon^2(u_r-w_r)\dd x\dd r\right)}_{L^{m/2}(\Omega)}=0.
\]
Assume moreover the setting of Assumption \ref{fixing null set}. Then we have the bound
\begin{equation}
\label{a priori sewing}
    \norm{\left(\int_0^t \int_\mathbb{T}\Sigma_\epsilon^2(u_r-w_r)\dd x\dd r\right)}_{L^{m/2}(\Omega)}\lesssim \norm{\Sigma_\epsilon^2}_{L^p_x}\|L\|_{C^{\gamma_1}_tW^{1,p'}_x}(1+\norm{u}_{C^{\gamma_0/2}_tL^m_\omega L^2_x}),
\end{equation}
where $p'$ is the Young convolutional conjugate of $p$. 
\end{lem}
\begin{proof}
We first verify that the germ $A^\epsilon$ does admit a sewing (remark that we do not require regularization from the local time in this setting as $\sigma_\epsilon$ is assumed to be smooth and bounded). Indeed, from \eqref{mollified diffusion}, using the simple identity $a^2-b^2=(a+b)(a-b)$, we have 
\begin{align*}
    \norm{(\delta A^\epsilon)_{s,u,t}}_{L^{m/2}(\Omega)}&=\norm{\int_u^t\norm{\sigma_\epsilon(u_u-w_r)}_{\cL_2(L^2_x)}^2-\norm{\sigma_\epsilon(u_s-w_r)}_{\cL_2(L^2_x)}^2 \dd r}_{L^{m/2}(\Omega)}\\
    &\leq 2c_\epsilon C_\epsilon\norm{ \norm{u_u-u_s}_{L^2_x} (t-u)}_{L^{m/2}(\Omega)}
    \\
    &\leq 2c_\epsilon C_\epsilon \| u\|_{C^{\gamma_0/2}_t L^{m}_\omega L^2_x}  |t-s|^{1+\gamma_0/2}.
\end{align*}
By application of the Sewing Lemma (e.g. \cite{frizhairer}), it follows that 
\[
\| (\cI A^\epsilon)_{s,t}-A^\epsilon_{s,t}\|\lesssim c_\epsilon C_\epsilon \| u\|_{C^{\gamma_0/2}_t L^{m}_\omega L^2_x}  |t-s|^{1+\gamma_0/2}.
\]
Furthermore, the germ
$
\tilde{A}^\epsilon_{s,t}=\int_s^t \norm{\sigma_\epsilon( u_r-w_r)}_{L^2_x}^2\dd r,
$
trivially admits a sewing as $\delta \tilde{A}^\epsilon=0$ and therefore $(\mathcal{I}\tilde{A}^\epsilon)_{s,t}=\tilde{A}^\epsilon_{s,t}$. We also observe that 
\begin{multline}
    \norm{A^{\epsilon}_{s,t}-\tilde{A}^{\epsilon}_{s,t}}_{L^{m/2}(\Omega)}
    =\norm{\int_s^t\norm{\sigma_\epsilon(u_s-w_r)}_{\cL_2(L^2_x)}^2-\norm{\sigma_\epsilon(u_r-w_r)}_{\cL_2(L^2_x)}^2\dd r}_{L^{m/2}(\Omega)}
    \\
    \leq 2c_\epsilon C_\epsilon \| u\|_{C^{\gamma_0/2}_t L^{m}_\omega L^2_x}  |t-s|^{1+\gamma_0/2}.
\end{multline}
Combining our two estimates from above, this allows us to deduce that 
\begin{align*}
     \norm{(\mathcal{I}A^\epsilon)_{s,t}-\left(\int_s^t \norm{\sigma_\epsilon( u_r-w_r)}_{L^2_x}^2\dd r\right)}_{L^{m/2}(\Omega)}
      &\leq \norm{(\mathcal{I}A^\epsilon)_{s,t}-A^\epsilon_{s,t}}_{L^{m/2}(\Omega)}+\norm{A^\epsilon_{s,t}-\tilde{A}^\epsilon_{s,t}}_{L^{m/2}(\Omega)}
     \\
     & \lesssim c_\epsilon C_\epsilon \| u\|_{C^{\gamma_0/2}_t L^{m}_\omega L^2_x}  |t-s|^{1+\gamma_0/2}.
\end{align*}
Hence, the function
$
t\to \norm{ \int_0^t \norm{\sigma_\epsilon(X_r-w_r)}_{\cL_2(L^2_x)}^2\dd r-(\mathcal{I}A^\epsilon)_{t}}_{L^{m/2}(\Omega)},
$
is a constant, starting in zero, which leads to our conclusion. Towards the second point, note that under the additional Assumption \ref{fixing null set}, we apply the Young convolution inequality in combination with elementary estimates to obtain the alternative bound 

\[
A^\epsilon_{s,t}=\int_s^t\int_\mathbb{T}\Sigma_\epsilon^2(u_s-w_r)\dd x\dd r=\int_\mathbb{T}(\Sigma_\epsilon^2*L_{s,t})(u_s)\dd x\lesssim \norm{\Sigma_\epsilon^2*L_{s,t}(u_s)}_{L^\infty_x}\lesssim \norm{\Sigma_\epsilon^2}_{L^p_x}|t-s|^{\gamma_0},
\]

as well as 
\begin{align*}
    \norm{(\delta A^\epsilon)_{s,u,t}}_{L^{m/2}(\Omega)}&=\norm{\int_\mathbb{T}\left( (\Sigma_\epsilon^2*L_{u,t})(u_s)-(\Sigma_\epsilon^2*L_{u,t})(u_u)\right)\dd x}_{L^{m/2}(\Omega)}\\&\lesssim \norm{\Sigma_\epsilon^2*L_{s,t}}_{C^1_x}\norm{\int_\mathbb{T}|u_s-u_u|\dd x}_{L^{m/2}(\Omega)}\\&\lesssim \norm{\Sigma_\epsilon^2}_{L^p_x}\| L\|_{C^{\gamma_1}W^{1,p'}}|t-s|^{\gamma_1}\norm{u}_{C^{\gamma_0/2}_tL^m_\omega L^2_x}|t-s|^{\gamma_0/2}, 
\end{align*}
where we have used the Jensen's inequality, and $p'$ is the Young conjugate of $p$. 
From this we directly infer \eqref{a priori sewing}.
\end{proof}
In the next Lemma, we show that the robustified a-priori bound \eqref{a priori sewing} can serve to extend the definition of the stochastic integral to singular diffusion coefficients $\sigma$ with the property that $\Sigma^2\in L^p_x$.

\begin{lem}
Let $\sigma$ be of the form \eqref{form of sigma} and $\sigma_\epsilon$ denote a corresponding cut-off mollification given by \eqref{cut-off mollification}. Suppose the setting of Assumption \ref{fixing null set} holds. Let $u$ and $\mathcal{I}A^\epsilon$ be as in Lemma \ref{identification} above and assume additionally that $\gamma_0$ satisfies \eqref{maximal gamma}. Then it holds that 
\begin{equation}
    \mathbb{E}\left[ \norm{\int_0^t \sigma_\epsilon( u_r-w_r)\dd W_r}_{L^2_x}^2\right]=\mathbb{E}\left[ \int_0^t \norm{\sigma_\epsilon(u_r-w_r)}_{\cL_2({L^2_x})}^2\dd r\right]= \norm{(\mathcal{I}A^\epsilon)_t}_{L^1(\Omega)},
    \label{ito iso}
\end{equation}
and for $m\geq 2$ the following version of the Burkholder-Davis-Gundy inequality holds
\begin{equation}
    \mathbb{E}\left[ \sup_{t\in [0,T]}\norm{ \int_0^t \sigma_\epsilon( u_r-w_r)\dd W_r}_{L^2_x}^m\right]\lesssim\norm{(\mathcal{I}A^\epsilon)_T}_{L^{m/2}(\Omega)} .
    \label{BDGII}
\end{equation}
In particular, the sequence $\left(\int_0^t \sigma_\epsilon( u_r-w_r)\dd W_r\right)_\epsilon$ is Cauchy in $L^{m}(\Omega, C([0,T], L^2_x))$, whose limit we denote by 
\[
t\to I_t\sigma(u-w)=\int_0^t \sigma( u_r-w_r)\dd W_r.
\]
By construction, we have the It\^o isometry
\[
 \mathbb{E}\left[ \norm{\int_0^t \sigma( u_r-w_r)\dd W_r}_{L^2_x}^2\right]=\norm{(\mathcal{I}A)_t}_{L^1(\Omega)}, \qquad where \qquad
A_{s,t}:=\int_\mathbb{T}(\Sigma^2*L_{s,t})(u_s)\dd x.
\]
The construction is independent of the chosen cut-off mollification and is adapted to the filtration generated by $(u, W)$.  Moreover the so constructed integral is linear in the sense that for two  functions $\sigma_1, \sigma_2$, we have
\[
I_t(\sigma_1+\sigma_2)(u-w)=I_t\sigma_1(u-w)+I_t\sigma_2(u-w).
\]
Finally, we have the a-priori bound
\[
 \mathbb{E}\left[ \sup_{t\in [0,T]}\norm{\int_0^t \sigma( u_r-w_r)\dd W_r}_{L^2_x}^m\right]\lesssim \norm{\Sigma^2}_{L^p_x}\|L\|_{C^{\gamma_1}_tW^{1,p'}_x}(1+\norm{u}_{L^{\gamma_0/2}_tL^{m/2}_\omega L^2_x}), 
\]
where $p'$ is the Young convolutional conjugate of $p$. 
\label{extending the integral}
\end{lem}
\begin{proof}
The above \eqref{ito iso} and \eqref{BDGII} are immediate consequences of the classical It\^o isometry and Burkholder-Davis-Gundy inequality available in this setting as well as the previous Lemma \ref{identification}. Moreover, for $\epsilon, \epsilon'>0$, we have similarly
\[
\mathbb{E}\left[ \sup_{t\in [0,T]}\norm{ \int_0^t (\sigma_\epsilon( u_r-w_r)-\sigma_{\epsilon'}( u_r-w_r))\dd W_r}_{L^2_x}^m\right]\lesssim\norm{(\mathcal{I}A^{\epsilon, \epsilon'})_T}_{L^{m/2}(\Omega)} 
\]
where $(\mathcal{I}A^{\epsilon, \epsilon'})$ denotes the Sewing of the germ
\[
A^{\epsilon, \epsilon'}_{s,t}=\int_{\mathbb{T}}\sum_k((\sigma_{k, \epsilon}-\sigma_{k, \epsilon'})^2*L_{s,t})(u_s)\dd x
\]
For notational ease, we define $\Sigma^2_{\epsilon, \epsilon'}=\sum_k(\sigma_{k, \epsilon}-\sigma_{k, \epsilon'})^2$. Note that by Vitali's convergence theorem $\norm{\Sigma^2_{\epsilon, \epsilon'}}_{L^p_x}\to 0$.
Recalling the available regularity of the averaging operator in \eqref{regularity available} of Assumption \ref{fixing null set}, we obtain immediately by Young's convolutional inequality 
\begin{align*}
    |A^{\epsilon, \epsilon'}_{s,t}|\lesssim \norm{\Sigma^2_{\epsilon, \epsilon'}*L_{s,t}}_{L^\infty_x}
    \lesssim \norm{\Sigma^2_{\epsilon, \epsilon'}}_{L^p_x}\|L\|_{C^{\gamma_1}_tL^{p'}_x}|t-s|^{\gamma_1},
\end{align*}
where $p'$ is the Young convolution conjugate of $p$. In addition, it is readily seen that 
\begin{align*}
\begin{aligned}
    |(\delta A^{\epsilon, \epsilon'})_{s,u, t}|&\leq \int_{\mathbb{T}}|(T_{u,t}^{-w}\Sigma^2_{\epsilon, \epsilon'})(u_s)-(T_{u,t}^{-w}\Sigma^2_{\epsilon, \epsilon'})(u_u)|\dd x
   \\& \lesssim\norm{(T_{u,t}^{-w}\Sigma^2_{\epsilon, \epsilon'})}_{C^1_x}\norm{u_u-u_s}_{L^2_x}\lesssim \norm{\Sigma^2_{\epsilon, \epsilon'}}_{L^p_x}\|L\|_{C^{\gamma_1}_tW^{1,p'}_x}|t-u|^{\gamma_1}\norm{u_u-u_s}_{L^2_x}.
    \end{aligned}
\end{align*}
 We have by Jensen's inequality that 
\begin{align*}
    \norm{(\delta A^{\epsilon, \epsilon'})_{s,u,t} }_{L^{m/2}(\Omega)}   \lesssim  \norm{(\delta A^{\epsilon, \epsilon'})_{s,u,t} }_{L^{m}(\Omega)} 
    \lesssim \norm{\Sigma^2_{\epsilon, \epsilon'}}_{L^p_x}\|L\|_{C^{\gamma_1}_tW^{1,p'}_x}|t-u|^{\gamma_{1}} |u-s|^{\gamma_0/2}\norm{u}_{C^{\gamma_0/2}_tL^m_\omega L^2_x}
\end{align*}
As by Assumption \ref{fixing null set} $\gamma_{1}+\gamma_0/2>1$, the above shows that $A^{\epsilon, \epsilon'}$ admits a  sewing $\mathcal{I}A^{\epsilon, \epsilon}$ for which we have 
\begin{align*}
    \norm{(\mathcal{I}A^{\epsilon, \epsilon'})_T}_{L^{m/2}(\Omega)}
    &\lesssim_T \norm{A^{\epsilon, \epsilon'}}_{C^{\gamma_0/2}_t L^{m/2}(\Omega)}+\norm{(\delta A^{\epsilon, \epsilon'})}_{C^{\gamma_0/2+\gamma_1}_tL^{m/2}(\Omega)}
    \\
    &\lesssim_T \norm{\Sigma^2_{\epsilon, \epsilon'}}_{L^p_x}\|L\|_{C^{\gamma_1}_tW^{1,p'}_x}\norm{u}_{C^{\gamma_0/2}_tL^m_\omega L^2_x}.
\end{align*}
We conclude that the sequence $\left( \int_0^t \sigma_\epsilon( u_r-w_r)\dd W_r\right)_\epsilon$ is Cauchy in $L^{m}(\Omega, C_t L^2_x)$, allowing to define the corresponding limit as the stochastic integral. Remark moreover that this construction is independent of the sequence of chosen cut-off mollifications, which is immediate by replacing $\sigma_{\epsilon'}$ by $\sigma$ in the above considerations. Adaptedness follows from the fact that the sequence of approximations is adapted by classical It\^o theory. Linearity follows from the fact that cut-off mollifications are linear, i.e. $(\sigma_1+\sigma_2)_\epsilon=\sigma_{1,\epsilon}+\sigma_{2,\epsilon}$ as well as the fact that the classical It\^o integral is linear. The last assertion follows as in Lemma \ref{identification}.
\end{proof}

   \section{Regularity of averaged fields and Volterra sewing}\label{sec:Volterra sewing}
As is common in the study of (semi-linear) SPDEs, we will crucially rely on some space-time regularity trade-offs. In the sequel we let $P:[0,T]\rightarrow \cL(L^2_x)$ denote the heat-semigroup. This operator plays a crucial role, as it allows us to gain spatial regularity at the cost of additional time singularities of Volterra type appearing in the integral. In order to benefit from such regularity trade-offs, we need to understand how singularities of Volterra type can be treated in the robustified sewing setting  we are concerned with here. This is precisely the content of this section. We will mainly invoke tools developed in \cite{Harang_Tindel21} to accommodate the Volterra structure. See also \cite{catellier2021pathwise} where a similar construction has been used in the case of space-(only)-noise. 
 The following lemma is an adaption of the Volterra sewing lemma from \cite{Harang_Tindel21} to the non-linear Young setting and to the specific setting in the current article. 

           \begin{lem}[Non-linear Young-Volterra integral]\label{lem: NLY volterra sewing}
 Let $E$ be a Banach space. Let $\eta\in (0,1)$ and  suppose  $A:[0,T]^2 \rightarrow E$ is such that   for $\gamma>\eta$ and $\alpha,\rho\in(0,1)$ satisfying  $\gamma+\alpha\rho>1$, then we have 
  \begin{equation}
      \begin{aligned}
          \|A\|_{\gamma}&:=\sup_{s,t\in [0,T]}\frac{\|A_{s,t}\|_E}{|t-s|^{\gamma}}<\infty 
          \\
          \|\delta A\|_{\gamma+\rho \alpha} &:=\sup_{s<u<t\in [0,T]}\frac{\|\delta A_{s,u,t}\|_E}{|t-s|^{\gamma+\alpha \rho}}<\infty . 
      \end{aligned}
  \end{equation}
 Then the Volterra Non-linear Young integral defined by 
    \begin{equation}
         \int_0^t (t-r)^{-\eta} A_{\dd r}:= \lim_{|\cP|\rightarrow 0} \sum_{[u,v]\in \cP} (t-u)^{-\eta}A_{u,v}
    \end{equation}
    is an element of $C^{\gamma-\eta-\delta}_t E$ for all $\delta>0 $. Furthermore, for $0<s<t<T$ and $\delta>0 $ the following bound holds: 
    \begin{equation}%\label{eq:NLY Volterra estimates}
    \begin{aligned}
                  \| \int_s^t (t-r)^{-\eta} A_{\dd r}\|_E&\lesssim (t-s)^{\gamma-\eta-\delta} (\|A\|_{\gamma}+\|\delta A\|_{\gamma+\alpha \rho})
    \end{aligned}
    \end{equation}
    Suppose in addition that for any $t\in [0,T]$, the limit $
    \lim_{\epsilon\to 0}\frac{1}{\epsilon}A_{t,t+\epsilon}$
    exists as an element in $E$, in which case we denote it by $(\partial_tA)_t$. If moreover $t\to (\partial_tA)_t$ is continuous we have 
    \[
   \int_0^t(t-r)^{-\eta}A_{\dd r}= \int_0^t(t-r)^{-\eta}(\partial_tA)_r\dd r.
    \]
    \end{lem}

 \begin{proof}
    This proof is a simple application of \cite[Lem. 22]{Harang_Tindel21} to our specific non-linear Young integrand $A$. We therefore try to keep the proof short and advice the reader to consult the reference for further details of the lemma. We introduce a new parameter $t\leq \tau\leq T$ and define 
    $$ A^\tau_{s,t}=(\tau-s)^{-\eta} A_{s,t}.
    $$
    According to \cite[Lem. 22]{Harang_Tindel21}, we first need to check that 
    \begin{equation}
    \begin{aligned}
        \|A_{s,t}^\tau\|_E&\lesssim |\tau-t|^{-\eta} |t-s|^{\gamma} \wedge |\tau-s|^{\gamma-\eta} 
        \\
        \| (A^{\tau}-A^{\tau'})_{s,t} \|_E &\lesssim |\tau-\tau'|^\theta |\tau'-t|^{-\eta-\theta}|t-s|^{\gamma}\wedge |\tau-s|^{\gamma -\eta}
\end{aligned}
    \end{equation}
    The first bound follows immediately from definition of $A^\tau_{s,t}=(\tau-s)^{-\eta}A_{s,t}$ and the assumed bound on $A$ together with the fact that $|\tau-s|^{-\eta}\leq |\tau-t|^{-\eta}$ and that $|t-s|^\gamma \leq |\tau-s|^\gamma$. 
  
    For the second estimate, we need to use that for any $\theta\in [0,1]$ we have 
    \begin{equation}\label{eq:singular diff}
        |(\tau-s)^{-\gamma}-(\tau'-s)^{-\gamma}|\lesssim |\tau-\tau'|^{\theta}|\tau'-s|^{-\eta-\theta},
    \end{equation}
and the estimate follows in the same way as for the first. 
Next we need to check that for $s\leq u \leq t $ there exists $\beta>1$ and $\kappa\in (0,1)$ such that for any $\theta\in [0,1]$ we also have 
        \begin{equation}\label{eq:delta cond}
    \begin{aligned}
        \|\delta A_{s,u,t}^\tau\|_E&\lesssim |\tau-t|^{-\kappa} |t-s|^{\beta} \wedge |\tau-s|^{\beta-\kappa} 
        \\
        \|\delta (A^{\tau}-A^{\tau'})_{s,u,t} \|_E &\lesssim |\tau-\tau'|^\theta |\tau'-t|^{-\kappa-\theta}|t-s|^{\beta}\wedge |\tau-s|^{\beta -\kappa}
\end{aligned}
    \end{equation}
These estimates will again follow by some simple bounds related to the singularity $(\tau-s)^{-\eta}$ together with the assumed bound on $\delta A_{s,u,t}$. More precisely, it is readily checked that 
\begin{equation*}
    \delta A_{s,u,t}^{\tau} = [(\tau-s)^{-\eta}-(\tau-u)^{-\eta}]A_{u,t}+(\tau-s)^{-\eta} \delta A_{s,u,t}. 
\end{equation*}
We will bound the two terms on the right hand side separately, and begin with the first term. There we use \eqref{eq:singular diff}  to see that 
\begin{equation*}
    \|[(\tau-s)^{-\eta}-(\tau-u)^{-\eta}]A_{u,t}\|_E\lesssim |u-s|^\theta |\tau-u|^{-\eta-\theta} \|A\|_\gamma |t-u|^\gamma. 
\end{equation*}
while for the second term we have 
\begin{equation*}
    \|(\tau-s)^{-\eta} \delta A_{s,u,t}\|_E\lesssim |\tau-s|^{-\eta}|t-s|^{\gamma+\rho\alpha}. 
\end{equation*}
Therefore, choosing $\theta = \rho \alpha$ we see that 
\begin{equation*}
   \|  \delta A_{s,u,t}^{\tau}\|_E \lesssim_T |\tau-u|^{-\eta-\rho\alpha}|t-s|^{\gamma+\rho\alpha}\wedge |\tau-s|^{\gamma-\eta}. 
\end{equation*}
Thus, for $\beta=\gamma+\rho\alpha>1$ and $\kappa=\eta+\rho\alpha$ we see that the first bound in \eqref{eq:delta cond} is satisfied. Furthermore, it is clear that $\beta-\kappa = \gamma-\eta$. 

For the second inequality in \eqref{eq:delta cond} we apply the very similar techniques, but now invoking the following bound for the singular kernel with $s\leq u \leq  t\leq \tau'\leq \tau$ we have for any $\theta,\zeta \in [0,1]$
\begin{equation*}
    |(\tau-t)^{-\eta}-(\tau'-t)^{-\eta}-(\tau-s)^{-\eta}+(\tau'-s)^{-\eta}|\lesssim |\tau-\tau'|^\theta |\tau'-t|^{-\eta-\theta-\zeta} |t-s|^{\zeta}. 
\end{equation*}
With this inequality, using again that 
\begin{multline*}
     \delta (A^{\tau}-A^{\tau'} )_{s,u,t}
     \\
     = [(\tau'-s)^{-\eta}-(\tau-s)^{-\eta}-(\tau'-u)^{-\eta}-(\tau-u)^{-\eta}]A_{u,t}+[(\tau-s)^{-\eta}-(\tau'-s)^{-\eta}] \delta A_{s,u,t}. 
\end{multline*}
we conclude by following the exact same steps as for the first bound in \eqref{eq:delta cond}. 
We conclude that the conditions in \eqref{eq:delta cond} holds with $\beta=\gamma+\alpha \rho$ and $\kappa = \eta+\alpha \rho$, and it follows by \cite[Lem. 22]{Harang_Tindel21} in combination with \cite[Rem. 19]{Harang_Tindel21} that $\int_0^t (t-r)^{-\eta} A_{\dd r}$ exists as an element of $C^{\gamma-\eta-\delta}_t E$ for any small $\delta>0$. Finally, it identification of the Volterra sewing with a Volterra integral in the case of a differentiable germ $A$ is obtained analogously to Lemma \ref{identification}.
    \end{proof}

\begin{cor}
\label{corrected coro}
Let $\sigma$ be of the form \eqref{form of sigma} and $\sigma_\epsilon$ denote a corresponding cut-off mollification given by \eqref{cut-off mollification}. Suppose the setting of Assumption \ref{fixing null set} holds. Let $u$ and $\mathcal{I}A^\epsilon$ be as in Lemma \ref{identification} above and assume additionally that $\gamma_0$ satisfies \eqref{maximal gamma}.  Then for any $\gamma_0,\eta\geq 0$ satisfying  $\gamma_0-\eta>\delta>0$ for some $\delta $ the following inequality holds: 
\begin{equation}\label{eq:bounds for integral}
\begin{aligned}
  \norm{\int_s^t (t-r)^{-\eta} \int_{\TT} \Sigma_\epsilon^2(u_r-w_r) \dd x\dd r}_{L^{m/2}(\Omega)}& \lesssim (t-s)^{\gamma_0-\eta-\delta}\norm{\Sigma_\epsilon^2}_{L^p_x} (1+\norm{u}_{C^{\gamma_0/2}_tL^m_\omega L^2_x}),
\end{aligned}
\end{equation}
\end{cor}
\begin{proof}
As observed in the proof of Lemma \ref{identification} already, the germ
   \begin{equation*}
       A_{s,t}^\epsilon  := \int_s^t \int_{\TT} \Sigma_\epsilon^2 (u_s-w_r)\dd x\dd r,
   \end{equation*}
admits a sewing in $E=L^{m/2}(\Omega)$ and moreover  
\[
\norm{A^\epsilon}_{\gamma_0/2}+\norm{\delta A^\epsilon}_{\gamma_0/2+\gamma_1}\lesssim \norm{\Sigma_\epsilon^2}_{L^p_x}\|L\|_{C^{\gamma_1}_tW^{1,p'}_x}(1+\norm{u}_{C^{\gamma_0/2}_tL^m_\omega L^2_x}).
\]
It then follows directly by Lemma \ref{lem: NLY volterra sewing} that the corresponding non-linear Volterra integral exists and enjoys the bound\[
\norm{\int_s^t(t-r)^{-\eta}A_{\dd r}^\epsilon}_{L^{m/2}(\Omega)}\lesssim (t-s)^{\gamma_0-\eta-\delta}\norm{\Sigma_\epsilon^2}_{L^p_x}\|L\|_{C^{\gamma_1}_tW^{1,p'}_x}(1+\norm{u}_{C^{\gamma_0/2}_tL^m_\omega L^2_x})
\]
Moreover, since $\Sigma^2_\epsilon$ is smooth,
\[
(\partial_tA^\epsilon)_t=\int_\mathbb{T}\Sigma^2_\epsilon(u_t-w_t)\dd x
\]
and therefore, by the last part of Lemma \ref{lem: NLY volterra sewing}
\begin{align*}
    \norm{\int_s^t (t-r)^{-\eta} \int_{\TT} \Sigma_\epsilon^2(u_r-w_r) \dd x\dd r}_{L^{m/2}(\Omega)}&=\norm{\int_s^t (t-r)^{-\eta} A_{\dd r}^\epsilon}_{L^{m/2}(\Omega;\RR)}
\end{align*}
concluding the claim. 

\end{proof}

 \section{Tightness}\label{sec:apriori and tightness}

     Combing the stochastic integral considerations of the two previous sections, we are now in place to establish a-priori bounds for solutions to the mollified problem. More precisely, let $\sigma$ be of the form \eqref{form of sigma} such that $\Sigma^2\in L^p_x$ and let $\sigma_\epsilon$ be a cut-off mollification as in \eqref{cut-off mollification}. Since $\sigma_\epsilon$ satisfies \eqref{mollified diffusion} for any $\epsilon>0$ fixed, the equation
     \begin{equation}
u^\epsilon_t=P_tu_0+\int_0^tP_{t-s}\sigma_\epsilon(u^\epsilon_s-w_s)\dd W_s, \qquad u_0\in H^{1/2}_x
         \label{mollified problem}
     \end{equation}
    classically admits a unique solution $u^\epsilon \in C^{1/2-\delta}_tL^m_\omega L^2_x$ for any $\delta>0$ and $m\in [1, \infty)$ (refer for example to \cite{cf:DpZ}). 

Towards the aim of proving existence of the limiting solution to \eqref{mollified problem}, we must establish two distinct a-priori bounds that will play central parts in the subsequent application of the Aubin-Lions lemma. These two bounds follow in the two next subsections.  
    
    \subsection{A first a priori bound}\label{sec:first apriori}
    Using the machinery developed in the previous sections, we first establish an a-priori bound for $u^\epsilon$ in $C^{\gamma_0/2}_tL^m_\omega L^2_x$ uniformly in $\epsilon>0$. 

 \begin{lem}
 \label{first a priori}
Let $\sigma$ be of the form \eqref{form of sigma} such that $\Sigma^2\in L^p_x$ and $\sigma_\epsilon$ denote a corresponding cut-off mollification given by \eqref{cut-off mollification}. Suppose the setting of Assumption \ref{fixing null set} holds. Let $u^\epsilon$ denote the unique solution to \eqref{mollified problem} with $u_0\in H^{\gamma_0/2}_x$. Then for $\gamma_0$ satisfying \eqref{maximal gamma}, we have for all $m\geq 2$
 \begin{equation*}
       \norm{u^\epsilon}_{C^{\gamma_0/2}_tL^m_\omega L^2_x}^m\lesssim \norm{u_0}^m_{H^{\gamma_0/2}_x}+\norm{\Sigma^2}_{L^p_x}^m.
 \end{equation*}
 \end{lem}
 \begin{proof}
As $u^\epsilon$ solves \eqref{mollified problem}, we write the equation on its mild form and obtain 
 \begin{equation}\label{eq:mild form proof}
 u^\epsilon_t- u^\epsilon_s=(P_t-P_s)u_0+\int_s^t P_{t-r}\sigma^\epsilon( u^\epsilon_r+w_r)\xi \dd r+\int_0^s(P_{t-s}-1)P_{s-r}\sigma^\epsilon( u^\epsilon_r+w_r)\xi_r\dd r.
 \end{equation}
 Concerning the first term involving the initial condition, we have by Lemma \ref{schauder} in the appendix that 
  \begin{align*}
        \norm{(P_t-P_s) u_0}_{L^2_x}\lesssim \norm{(-\Delta)^{-\gamma_0/2}(P_t-P_s)}_{\cL(L^2_x)}\norm{u_0}_{H^{\gamma_0/2}_x}\lesssim (t-s)^{\gamma_0/2} \norm{u_0}_{H^{\gamma_0/2}_x}.
    \end{align*}
 Concerning the middle term of \eqref{eq:mild form proof}, we have by It\^o's isometry in the Hilbert space $L^2_x$
 \begin{equation}
     \begin{split}
         \mathbb{E}\norm{\int_s^t P_{t-r}\sigma^\epsilon(u^\epsilon_r-w_r)\dd W_r}_{L^2_x}^m&=\mathbb{E}\left[\left(\int_s^t\norm{P_{t-r} \sigma^\epsilon(u^\epsilon_r-w_r)}^2_{\cL_2( L^2_x)}\dd r\right)^{m/2}\right]\\
         &\leq 
    \mathbb{E}\left[\left(\int_s^t\norm{ \sigma^\epsilon(u^\epsilon_r-w_r)}^2_{\cL_2(L^2_x)}\dd r\right)^{m/2}\right]\\
    &=\mathbb{E}\left[\left(\int_s^t\int_\mathbb{T}\Sigma^2_\epsilon(u^\epsilon_r-w_r)\dd x\dd r\right)^{m/2}\right].
    \end{split}
    \label{regularity trade off}
 \end{equation}
Concerning the last term of \eqref{eq:mild form proof}, we proceed similarly using additionally  the Schauder estimate of Lemma \ref{schauder} in the Appendix,  we have 
\begin{equation}
    \begin{split}
    &\mathbb{E}\norm{\int_0^s(P_{t-r}-P_{s-r})\sigma^\epsilon(u^\epsilon_r-w_r)\dd W_r}^m_{L^2_x}\\
    &\lesssim \mathbb{E}\left[\left(\int_0^s \norm{(P_{t-r}-P_{s-r})}_{\cL(L^2_x)}^2\norm{\sigma^\epsilon(u^\epsilon_r-w_r)}_{\cL_2(L^2_x)}^2\dd s\right)^{m/2}\right]\\ 
    &\lesssim (t-s)^{\gamma_0 m/2}\mathbb{E}\left[\left(\int_0^s (s-r)^{-\gamma_0}\norm{\sigma^\epsilon(u^\epsilon_r-w_r)}_{\cL_2(L^2_x)}^2\dd s\right)^{m/2}\right]\\
    &=(t-s)^{\gamma_0 m/2}\mathbb{E}\left[\left(\int_0^s (s-r)^{-\gamma_0}\int_\mathbb{T}\Sigma^2_\epsilon(u^\epsilon_r-w_r)\dd x\dd s\right)^{m/2}\right].
        \label{regularity trade off II}
    \end{split}
\end{equation}
Combining our findings above shows that
\begin{equation}
\label{intermediate step}
     \begin{split}
          \mathbb{E}\norm{u^\epsilon_t-u^\epsilon_s}_{L^2_x}^m
        &\lesssim (t-s)^{\gamma_0 m/2}\norm{u_0}_{H^{\gamma_0/2}_x}^m+\mathbb{E}\left[\left(\int_s^t\int_{\mathbb{T}} \Sigma_\epsilon^2(u^\epsilon_r-w_r)\dd x\dd r\right)^{m/2} \right]\\
        &+(t-s)^{\gamma_0 m/2}\mathbb{E}\left[ \left(\int_0^s (s-r)^{-\gamma_0} \int_{\mathbb{T}} \Sigma_\epsilon^2(u^\epsilon_r-w_r)\dd x\dd r\right)^{m/2}\right]. \
     \end{split}
 \end{equation}
 Using the a-prioi bound \eqref{a priori sewing} obtained in the end of Lemma \ref{identification} we have
 \begin{align*}
 \mathbb{E}\left[\left(\int_s^t\int_{\mathbb{T}} \Sigma_\epsilon^2(u^\epsilon_r-w_r)\dd x\dd r\right)^{m/2} \right] &\lesssim |t-s|^{\gamma_0 m/2} \norm{\Sigma_\epsilon^2}_{L^p_x}^{m/2}(1+\norm{u^\epsilon}_{C^{\gamma_0/2}_tL^{m}_\omega L^2_x}^{m/2}).
 \end{align*}
  Moreover, using Corollary \ref{corrected coro}, we have that
  \begin{align*}
  \mathbb{E}\left[\left(\int_0^s(s-r)^{-\gamma_0} \int_{\mathbb{T}} \Sigma_\epsilon^2(u^\epsilon_r-w_r)\dd x\dd r\right)^{p/2} \right] &\lesssim \norm{\Sigma_\epsilon^2}_{L^p_x}^{m/2}\|L\|_{C^{\gamma_1}_tW^{1,p'}_x}(1+\norm{u^\epsilon}_{C^{\gamma_0/2}_tL^{m}_\omega L^2_x}^{m/2})
 \end{align*}
More precisely, we are allowed to use Corollary \ref{corrected coro} by exploiting that thanks to the strict inequality on $H$ in \eqref{condition in H} of Assumption \ref{fixing null set}, we also have 
\[
\norm{\Sigma^2_\epsilon\ast L_{s,t}}_{L^\infty_x}\lesssim \norm{\Sigma_\epsilon^2}_{L^p_x}|t-s|^{\gamma_0+\delta},
\]
for some small $\delta$, meaning the germ
\[
A^\epsilon_{s,t}=\int_s^t \int_\mathbb{T}\Sigma_\epsilon^2(u^\epsilon_s-w_r)\dd x\dd r,
\]
in Corollary \ref{corrected coro} actually enjoys local regularity $\gamma_0+\delta$, and thus enough to compensate the Volterra singularity of order $\gamma_0$ as demanded in Corollary \ref{corrected coro}.
Going back to \eqref{intermediate step}, we can therefore conclude that 
\begin{align*}
    \mathbb{E}\norm{u^\epsilon_t-u^\epsilon_s}^m_{L^2_x}\lesssim (t-s)^{\gamma_0 m/2}\left(\norm{u_0}^{m}_{H^{\gamma_0/2}_x}+\norm{\Sigma^2_\epsilon}_{L^p_x}^{m/2}(1+\norm{u^\epsilon}^{m/2}_{C^{\gamma_0/2}_t L^m_\omega L^2_x}) \right)
\end{align*}
which in particular implies that 
\begin{align*}
   \norm{u^\epsilon}^{m}_{C^{\gamma_0/2}_t L^m_\omega L^2_x}&\lesssim \norm{u_0}^{m}_{H^{\gamma_0/2}_x}+\norm{\Sigma^2_\epsilon}_{L^p_x}^{m/2}(1+\norm{u^\epsilon}^{m/2}_{C^{\gamma_0/2}_t L^m_\omega L^2_x}) \\
   &\lesssim 
  \norm{u_0}^{m}_{H^{\gamma_0/2}_x}+\norm{\Sigma^2}_{L^p_x}^{m/2}(1+\norm{u^\epsilon}^{m/2}_{C^{\gamma_0/2}_t L^m_\omega L^2_x}).
\end{align*}
where in the last step, we used lower semi continuity of the norm and the fact that $\Sigma^2_\epsilon\to \Sigma^2$ by Vitali's convergence theorem. The claim now follows directly. 
 \end{proof}
 \begin{rem}\label{rem:pathwise cont of first bound}
     Since Lemma \ref{first a priori} holds for any $m\geq 2$, we can combine it with the Kolmogorov continuity theorem to conclude that also $u^\epsilon$ is also pathwise continuous, taking values in $L^2_x$. Furthermore, we have  that  $\|u^\epsilon\|_{L_\omega^m C_t^{\gamma_0/2} L^2_x}<\infty$.
 \end{rem}

    \subsection{A second a-priori bound}
    In this section, we set out to establish a-priori bounds of the form 
    \[
    \mathbb{E}\norm{u^\epsilon}_{L^\infty_tH^{\gamma_0}_x}^m<\infty
    \]
    for $m$ sufficiently large. Towards this end, we exploit modified bounds on the stochastic convolution based on the estimates derived in \cite[Proposition 7.9]{cf:DpZ} 

    \begin{lem}
    \label{maximal inequality 2}
    Let $W$ be a cylindrical Wiener process on $L^2_x$, and let $H^\lambda_x$ denote the $L^2_x$ based fractional Sobolev spaces of regularity $\lambda$. Consider the stochastic convolution 
    \[
    W^\phi_t:=\int_0^tP_{t-s}\phi(s)\dd W_s, 
    \]
    for some $\phi\in C([0,T];\cL_2(L^2_x))$. 
    Then for $m\geq 2$ and $1/m<\alpha<1$ it holds that for any  $\lambda<1-(2\alpha)$  we have
    \begin{equation}\label{eq:lemma ineq second}
    \begin{split}
         \mathbb{E}\sup_{t\in [0, T]}\norm{W^\phi_t}^{m}_{H^\lambda_x}\leq c_{r, T}\int_0^T\mathbb{E}\left(\int_0^s(s-r)^{-2\alpha- \lambda}\norm{\phi(r)}^{2}_{\cL_2( L^2_x)}\dd r\right)^{m/2}\dd s.
     \end{split}
    \end{equation}
    \end{lem}
  \begin{rem}
     The gain of up to almost one derivative in space (expressed by the condition $\lambda<1$ obtained asymptotically for $m\to \infty$ ) is analogous to the stochastic convolution estimate obtained in \cite[Theorem 2.1]{Bechtold2021}, see also \cite{GERASIMOVICS2021109200}. 
     \end{rem}

    \begin{proof}
The proof follows along the lines of \cite[Theorem 2.1]{Bechtold2021}, but  for the sake of completeness let us sketch the proof to the above claim and explain the parameter regimes we obtain in our setting. The proof follows the classical factorization Lemma. Let  $r\leq s\leq t$ and $ 0<\alpha<1$. Starting from the identity 
    \begin{equation*}
    \int_{r}^t(t-s)^{\alpha-1}(s-r)^{-\alpha}\dd s=\frac{\pi}{\sin{\pi \alpha}},
\end{equation*}
we obtain, thanks to Fubini, the factorization 
\begin{equation*}
    W^\phi_t=\frac{\sin{\pi \alpha}}{\pi}\int_0^t(t-s)^{\alpha-1}P_{t-s}Y_s\dd s, \qquad \mathrm{ with} \qquad 
    Y_s=\int_0^s (s-r)^{-\alpha}P_{s-r}\phi(r)\dd W_r.
\end{equation*}
 Applying H\"older's inequality, we therefore have 
\begin{align}\label{eq:sup W}
    \sup_{t\in [0,T]}\norm{W^\phi_t}_{H^\lambda_x}^{m}\leq \sup_{t\in [0,T]} \left(\int_0^t(t-s)^{(\alpha-1)\frac{m}{m-1}}\dd s\right)^{1-1/m} \int_0^t \norm{Y_s}^{2r}_{H^\lambda_x}\lesssim \int_0^T \norm{Y_s}^{m}_{H^\lambda_x}\dd s,
\end{align}
where we exploited $m>1/\alpha$ in the last inequality. Concerning the remaining expression, we can make use of It\^o's isometry and the Schauder estimate \ref{schauder} to obtain
\begin{equation}
\begin{aligned}\label{eq:int Y}
    \mathbb{E}\int_0^T \norm{Y_s}^{m}_{H^\lambda_x}\dd s&= \mathbb{E}\int_0^T\norm{\int_0^s(s-r)^{-\alpha}P_{s-r}\phi(r)\dd W_r}_{H^\lambda_x}^{m}\dd s\\
    &\leq \mathbb{E}\int_0^T\left(\int_0^s (s-r)^{-2\alpha}\norm{P_{s-r}}_{\cL(L^2_x, H^\lambda_x)}^2\norm{\phi(r)}_{\cL_2(L^2_x)}^2\dd r\right)^{m/2}\dd s\\
    &\lesssim \int_0^T \mathbb{E}\left(\int_0^s (s-r)^{-2\alpha-\lambda}\norm{\phi(r)}_{\cL_2(L^2_x)}^2\dd r\right)^{m/2}\dd s.
\end{aligned}
\end{equation}
Combining the two estimates in \eqref{eq:sup W} and \eqref{eq:int Y} allows us to conclude that \eqref{eq:lemma ineq second} holds. 
    \end{proof}
    
  Invoking this Lemma, the a-priori bound is obtained in a straightforward fashion: Replace again the right hand side Lebesgue integral by a stochastic sewing that is robust and will be uniformly bounded thanks to the a-priori estimate obtained in Section  \ref{sec:first apriori}.

\begin{prop}
Let $\sigma$ be of the form \eqref{form of sigma} such that $\Sigma^2\in L^p_x$ and $\sigma_\epsilon$ denote a corresponding cut-off mollification given by \eqref{cut-off mollification}. Suppose the setting of Assumption \ref{fixing null set} holds. Let $u^\epsilon$ denote the unique solution to \eqref{mollified problem} with $u_0\in H^{\gamma_0}_x$. Then for $\gamma_0$ satisfying \eqref{maximal gamma}, we have for $m$ sufficiently large 
\begin{equation}
    \norm{u^\epsilon}_{L^m_\omega C_t H^{\gamma_0}_x}^m\lesssim \norm{u_0}_{H^{\gamma_0}_x}^m+\norm{\Sigma^2}^m_{L^p_x}.
\end{equation}
\end{prop}
\begin{proof}
    As $u^\epsilon$ solves the stochastic  integral equation 
    \[
u^\epsilon_t=P_tu_0+\int_0^tP_{t-s}\sigma_\epsilon(u^\epsilon_s-w_s)\dd W_s,
    \]
    we can apply Lemma \ref{maximal inequality 2} to obtain that for $\alpha\in (1/m, 1)$ the following bound holds
    \begin{align*}
        \mathbb{E}\sup_{t\in [0, T]}\norm{\int_0^tP_{t-s}\sigma_\epsilon(u^\epsilon_s-w_s)\dd W_s}^{m}_{H^{\gamma_0}_x}&\lesssim \int_0^T\mathbb{E}\left(\int_0^s (s-r)^{-2\alpha-\gamma_0}\int_{\mathbb{T}}\Sigma_\epsilon^2(u_r-w_r)\dd x\dd r\right)^{m/2}\dd s.
    \end{align*}
Similar to the proof of Lemma \ref{first a priori}, note that since the condition on $H$ in Assumption \ref{fixing null set} is strict, we obtain 
\[
\norm{T^{-w}\Sigma^2_\epsilon}_{L^\infty_x}\lesssim \norm{\Sigma_\epsilon^2}_{L^p_x}|t-s|^{\gamma_0+\delta},
\]
for some small $\delta$. This means in particular that the germ
\[
A^\epsilon_{s,t}=\int_s^t \int_\mathbb{T}\Sigma_\epsilon^2(u^\epsilon_s-w_r)\dd x\dd r,
\]
in Corollary \ref{corrected coro} actually enjoys local regularity of order $\gamma_0+\delta$. Provided $m$ is chosen sufficiently large such that we can choose $\alpha$ so small that $3\alpha<\delta$, this is enough to compensate the Volterra singularity of order $\gamma_0$ as demanded in Corollary \ref{corrected coro}. Overall, using Corollary \ref{corrected coro} as well as Lemma \ref{first a priori}, this allows to obtain the bound
\begin{align*}
     \mathbb{E}\sup_{t\in [0, T]}\norm{\int_0^tP_{t-s}\sigma_\epsilon(u^\epsilon_s-w_s)\dd W_s}^{m}_{H^{\gamma_0}_x}&\lesssim \norm{\Sigma^2_\epsilon}^{m/2}_{L^p_x}(1+\norm{u^\epsilon}_{C^{\gamma_0/2}_tL^m_\omega L^2_x}^{m/2})\int_0^Ts^{(\delta-3\alpha)m/2}\dd s\\
     &\lesssim_T \norm{u_0}^m_{H^{\gamma_0/2}_x}+\norm{\Sigma^2}^m_{L^p_x}.
\end{align*}
Concerning the initial condition, we have moreover 
    \begin{align*}
        \mathbb{E}\sup_{t\in [0, T]}\norm{P_tu_0}_{H^{\gamma_0}_x}^{m}\leq \norm{u_0}_{H^{\gamma_0}_x}^{m}
    \end{align*}
    which yields the claim. 
\end{proof}

    \subsection{Extraction of convergent subsequence}
    Summarizing the results of the two previous subsections, we have for sufficiently large $ m$  and $\gamma_0$ satisfying \ref{maximal gamma} that
    \begin{equation}
        \norm{u^\epsilon}_{C^{\gamma_0/2}_tL^m_\omega  L^2_x}^m+\norm{u^\epsilon}_{L^m_\omega  C^{\gamma_0/2}_tL^2_x}^m+\norm{u^\epsilon}_{L^m_\omega  L^\infty_tH^{\gamma_0}_x}^m\lesssim \norm{\Sigma^2}_{L^p_x}^m+\norm{u_0}^m_{H^{\gamma_0}_x}.
        \label{a priori for vitali}
    \end{equation}
    Recall that by the Aubin-Lions Lemma (refer for example to \cite[Proposition 5]{simon}), we have that the embedding 
    \[
    C^{\gamma_0/2}_tL^2_x\cap L^\infty_tH^{\gamma_0}_x \hookrightarrow C_tL^2_x,
    \]
    is compact. We therefore conclude that the sequence $(u^\epsilon)_\epsilon$ is tight in $C_tL^2_x$, meaning that by Prokhorov and Skorokhod we may conclude the following:
    \begin{cor}[Extraction of a convergent subsequence]
\label{extracted subsequence}
 Let $\sigma$ be of the form \eqref{form of sigma} such that $\Sigma^2\in L^p_x$ and $\sigma_\epsilon$ denote a corresponding cut-off mollification given by \eqref{cut-off mollification}. Suppose the setting of Assumption \ref{fixing null set}. Let $u^\epsilon$ denote the unique solution to \eqref{mollified problem} with $u_0\in H^{\gamma_0}_x$. There exists a probabilistic basis $(\bar{\Omega}, \bar{\mathcal{F}}, \bar{\mathbb{P}})$, processes $(\bar{u^\epsilon}, \bar{W^\epsilon})$ on the said basis whose laws coincide with those of $(u^\epsilon, W^\epsilon)$, and processes $(\bar{u}, \bar{W})$  such that 
\[
(\bar{u^\epsilon}, \bar{W^\epsilon})\to (\bar{u}, \bar{W}),
\]
$\bar{\mathbb{P}}$-almost surely in $C_t L^2_x\times C_t H^{-1-}_x$ along a subsequence $(\bar{u^{\epsilon_n}}, \bar{W^{\epsilon_n}})_n$ which we will in the following denote $(\bar{u^{\epsilon}}, \bar{W^{\epsilon}})_\epsilon$,  again by slight abuse of notation. Moreover, $\bar{W}$ and $\bar{W}^\epsilon$ are $(\bar{\Omega}, \bar{\mathcal{F}}, \Bar{\mathbb{P}})$-Brownian motions. Denote by $(\bar{\mathcal{F}}_t)_t$ the augmentation of the filtration generated by $(\bar{u}, \bar{W})$. For weak-* lower semi-continuity of norms, we have 
\begin{equation}
  \norm{\bar{u}}_{C^{\gamma_0/2}_tL^m_\omega L^2_x}^m+\norm{\bar{u}}_{L^m_\omega  C^{\gamma_0/2}_tL^2_x}^m+\norm{\bar{u}}_{L^m_\omega L^\infty_tH^{\gamma_0}_x}^m\lesssim \norm{\Sigma^2}_{L^p_x}^m+\norm{u_0}^m_{H^{\gamma_0}_x}.
    \label{a priori bounds limit}
\end{equation}

\end{cor}

    \section{Identification of the limit}\label{sec:identification}
  After obtaining a potential solution candidate $(\bar{u},\bar{W})$ in the previous Corollary \ref{extracted subsequence}, it remains to verify that the solution indeed satisfy the infinite dimensional equation
  \begin{equation}
       \bar{u}_t=u_0+\int_0^t \Delta \bar{u}_r\dd r+\int_0^t\sigma(\bar{u}_r-w_r)d\bar{W}_r.
       \label{limit eqn}
  \end{equation}
Let us recall that the stochastic integral on the right hand side of \eqref{limit eqn} is well-defined thanks to the a priori bound \eqref{a priori bounds limit} and Lemma \ref{extending the integral}. While 
   the convergence of the linear part is immediate, the main challenge will consist in establishing the convergence of the stochastic integral term, i.e. showing that 
   \[
   \int_0^t\sigma(\bar{u}^\epsilon_r-w_r)d\bar{W}^\epsilon_r\to \int_0^t\sigma(\bar{u}_r-w_r)d\bar{W}_r
   \]
   Note that as we are working in a robustified sewing setting of Lemma \ref{extending the integral} and not with classical $L^2_\omega L^2_t$ based It\^o theory, the martingale representation theorem usually employed at this stage is unavailable to us. Instead, we use the more direct approach of the stochastic compactness method (\cite[Lemma A.1]{hofmanova2013degenerate}, \cite{BreitFeireislHofmanova}) already adapted in \cite{bechtold} to the finite dimensional setting. In the following, we demonstrate that the arguments developed in \cite{bechtold} essentially also extend to our infinite dimensional  setting with some care. Throughout this section, let $(e_j)_{j}$ be an orthonormal basis of $L^2(\mathbb{T})$. One can readily check that the three processes 
\begin{equation}
    \begin{split}
         t &\to M^{j,\epsilon}_{t}:= \langle \left(\bar{u}^\epsilon_t-u_0-\int_0^t \Delta \bar{u}^\epsilon_r\dd r\right), e_{j}\rangle=\langle\int_0^t\sigma_\epsilon(\bar{u}^\epsilon_r-w_r)d\bar{W}_r, e_j\rangle,\\
    t &\to (M^{j,\epsilon}_{t})^{2}- \sum_k\int_0^t|\int_\mathbb{T}\sigma_{k,\epsilon}( \bar{u}^\epsilon_r-w_r) e_j\dd x |^2 \dd r,\\
   t &\to M^{j,\epsilon}_t  \langle \bar{W}^\epsilon_t, e_i\rangle -\int_0^t\langle \sigma_{i,\epsilon}( \bar{u}^\epsilon_r-w_r), e_j \rangle \dd r,
    \end{split}
    \label{start}
\end{equation}
%\begin{equation}
%    \begin{split}
%         M^{j,\epsilon}: t &\to \langle x_0, e_j\rangle+\langle e_j,  \int_0^t \sigma_\epsilon( \bar{X}^\epsilon_r-w_r)d\Bar{B}_r^\epsilon\rangle =\langle e_j, \bar{X}^\epsilon_t\rangle, \\
%    (M^2)^{j, \epsilon}: t &\to\langle \bar{X}^\epsilon_t, e_j\rangle ^2+\langle x_0, e_j\rangle^2-\int_0^t | \sigma_\epsilon^*( \bar{X}^\epsilon_r-w_r) e_j |^2\dd r,\\
%    (M\bar{B})^{i,j,\epsilon}:t &\to\langle M^\epsilon_t, e_j\rangle  \langle \bar{B}^\epsilon_t, f_i\rangle -\int_0^t\langle f_i,  \sigma_\epsilon^*( \bar{X}^\epsilon_r-w_r)e_j \rangle \dd r
%    \end{split}
%    \label{start}
%\end{equation}
are all martingales with respect to $(\bar{\mathcal{F}}_t)_t$. Note that since $\sigma_\epsilon$ is smooth, we have by Lemma~\ref{identification} that
\begin{align*}
\int_0^t\langle \sigma_{i,\epsilon}( \bar{u}^\epsilon_r-w_r), e_j \rangle \dd r=(\mathcal{I}a^{i,j,\epsilon})_t,
\end{align*}
where 
\begin{align*}
    a^{i,j,\epsilon}_{s,t}= \langle (\sigma_{i, \epsilon}*L_{s,t})(\bar{u}^\epsilon_s), e_j\rangle.
\end{align*}
Note that \eqref{start} being martingales is equivalent to having that for any bounded continuous functional $\phi$ on $C([0,s], L^2)\times C([0,s], H^{-1-\delta})$  then 
\begin{equation}
    \begin{split}
        \bar{\mathbb{E}}[\phi(\bar{u}^\epsilon|_{[0,s]}, \bar{W}^\epsilon|_{[0,s]}) (M^{j, \epsilon}_t-M^{j, \epsilon}_s)]&=0,\\
        \bar{\mathbb{E}}[\phi(\bar{u}^\epsilon|_{[0,s]}, \bar{W}^\epsilon|_{[0,s]}) ((M^{j, \epsilon}_{t})^2-(M^{j, \epsilon}_{s})^2-\sum_k\int_s^t|\int_\mathbb{T}\sigma_{k,\epsilon}( \bar{u}^\epsilon_r-w_r) e_j\dd x |^2 \dd r]&=0,\\
        \bar{\mathbb{E}}[\phi(\bar{u}^\epsilon|_{[0,s]}, \bar{W}^\epsilon|_{[0,s]}) (M^{j,\epsilon}_{t}\langle\bar{W}_{t}^{\epsilon},e_{i}\rangle-M^{j,\epsilon}_{s}\langle\bar{W}_{s}^{\epsilon},e_{i}\rangle-(\mathcal{I}a^{i,j,\epsilon})_{s,t})]&=0.
        \label{martingale_epsilon}
    \end{split}
\end{equation}
We next intend to pass to the limit in \eqref{martingale_epsilon}. Note that by almost sure convergence of Corollary \ref{extracted subsequence} and \eqref{a priori for vitali} all the terms with the exception of the appearing sewing and the term 
\begin{equation}
    \sum_k\int_s^t|\int_\mathbb{T}\sigma_{k,\epsilon}( \bar{u}^\epsilon_r-w_r) e_j\dd x |^2 \dd r,
    \label{critical term}
\end{equation}
converge due to  Vitali's convergence theorem. For the sewing $\mathcal{I}a^{i,j,\epsilon}$ we shall employ Lemma \ref{sewing convergence} from the appendix, which is illustrated in the next Lemma. 
\begin{lem}
\label{step one conv}
Suppose the setting of Corollary \ref{extracted subsequence}. For $s< t\in [0,T]$ and $m\geq 1$, we 
\begin{align*}
    \norm{(\mathcal{I}a^{\epsilon})_{s,t}- (\mathcal{I}a)_{s,t}}_{L^m(\bar{\Omega})}&\to 0,
\end{align*}
where 
\begin{align*}
    a^{i,j}_{s,t}&= \langle (\sigma_{i}*L_{s,t})(\bar{u}_s), e_j\rangle.
\end{align*}
\end{lem}
\begin{proof}

Observe that due to \eqref{a priori for vitali} we have 
\begin{align*}
    \norm{(\delta a^{i, j, \epsilon})_{s,u,t}}_{L^m(\bar{\Omega})}&=\norm{\int_\mathbb{T}((\sigma_{i, \epsilon}*L_{u,t})(\bar{u}^\epsilon_u)-(\sigma_{i, \epsilon}*L_{u,t})(\bar{u}^\epsilon_s))e_j\dd x}_{L^m(\bar{\Omega})}\\
    &\lesssim\bar{\mathbb{E}}\left[\sup_{t\neq s\in[0,T]}\frac{\norm{\bar{u}^\epsilon_{t}-\bar{u}^\epsilon_s}_H}{|t-s|^{\gamma_0/2}}^m\right]^{1/m}\norm{\sigma_{i, \epsilon}}_{L^p_x}\|L\|_{W^{1,p'}_x}|t-s|^{\gamma_0/2+\gamma_1}\\
    &\lesssim \norm{\sigma_i}_{L^p_x}|t-s|^{\gamma_0/2+\gamma_1}
\end{align*}
uniformly in $\epsilon>0$. Moreover, by \eqref{a priori for vitali} and Vitali's theorem, we have that actually $u^\epsilon\to u$ in $L^m_\omega C_t L^2_x$. We therefore observe that
\begin{align*}
    \norm{a^{i,j}_{s,t}-a^{i,j,\epsilon}_{s,t}}_{L^m(\bar{\Omega})}&\leq \norm{\int_\mathbb{T}((\sigma_i*L_{s,t})(\bar{u}_s)-(\sigma_i*L_{s,t})(\bar{u}^\epsilon_s))e_j\dd x}_{L^m(\bar{\Omega})}\\
    &+\norm{\int_\mathbb{T}((\sigma_i*L_{s,t})(\bar{u}_s^\epsilon)-(\sigma_{i, \epsilon}*L_{s,t})(\bar{u}^\epsilon_s))e_j\dd x}_{L^m(\bar{\Omega})}\\
    &\lesssim |t-s|^{\gamma_1}\norm{\sigma_i}_{L^p_x}\bar{\mathbb{E}}[\norm{\bar{u}^\epsilon-\bar{u}}_{C([0,T], H)}^m]^{1/m}+\norm{\sigma_i-\sigma_{i,\epsilon}}_{L^p}|t-s|^{\gamma_0/2}. 
\end{align*}
By Lemma \ref{sewing convergence}, this implies that indeed $\mathcal{I}a^{\epsilon}\to \mathcal{I}a$ in $C^{\gamma_0/2 \wedge \gamma_1}_tL^m(\bar{\Omega})$ and thus in particular the claim.
\end{proof}
We now pass to the convergence of \eqref{critical term}
\begin{lem}
\label{main convergence lemma}
The following convergence holds in $L^{m/2}(\bar{\Omega})$:
\[
\sum_k\int_s^t|\int_\mathbb{T}\sigma_{k,\epsilon}( \bar{u}^\epsilon_r-w_r) e_j\dd x |^2 \dd r \to \sum_k\int_s^t|\int_\mathbb{T}\sigma_{k}( \bar{u}_r-w_r) e_j\dd x |^2 \dd r.
\]

\end{lem}
\begin{proof}
    Let us first remark that the right hand side is a well-defined object in $L^{m/2}(\Omega)$. Indeed, this can be seen by 
    \begin{align*}
       \norm{ \sum_k\int_s^t|\int_\mathbb{T}\sigma_{k}( \bar{u}_r-w_r) e_j\dd x |^2 \dd r}_{L^m(\Omega)}&\lesssim\norm{ \sum_k\int_s^t\int_\mathbb{T}|\sigma_{k}( \bar{u}_r-w_r)|^2\dd x\dd r}_{L^m(\Omega)}\\&=\norm{\int_s^t\int_\mathbb{T}\Sigma^2( \bar{u}_r-w_r)\dd x\dd r}_{L^m(\Omega)}\\
        &\lesssim \norm{\Sigma^2}_{L^p_x}\norm{\bar{u}}_{L^m_\omega C^{\gamma_0/2}_tL^2_x},
    \end{align*}
 where we   again exploited the gain of regularity due to the local time of $w$. Next, note that by H\"older's inequality we have 
 \begin{align*}
     &\sum_k\int_s^t|\int_\mathbb{T}\sigma_{k,\epsilon}( \bar{u}^\epsilon_r-w_r) e_j\dd x |^2 \dd r-\sum_k\int_s^t|\int_\mathbb{T}\sigma_{k}( \bar{u}_r-w_r) e_j\dd x |^2 \dd r\\
     =&\sum_k\int_s^t \left(\int_\mathbb{T}\sigma_{k,\epsilon}( \bar{u}^\epsilon_r-w_r) e_j -\sigma_{k}( \bar{u}_r-w_r) e_j\dd x\right)\left( \int_\mathbb{T}\sigma_{k,\epsilon}( \bar{u}^\epsilon_r-w_r) e_j+\sigma_{k}( \bar{u}_r-w_r) e_j\dd x\right)\dd r\\
     &\lesssim \left(\sum_k\int_s^t  \left(\int_\mathbb{T}\sigma_{k,\epsilon}( \bar{u}^\epsilon_r-w_r) e_j -\sigma_{k}( \bar{u}_r-w_r) e_j\dd x\right)^2\dd r\right)^{1/2}\Xi^\epsilon_{s,t},
 \end{align*}
 where the term $\Xi$ is given by 
\[
\Xi^\epsilon_{s,t}=\left(\sum_k\int_s^t  \left(\int_\mathbb{T}\sigma_{k,\epsilon}( \bar{u}^\epsilon_r-w_r) e_j +\sigma_{k}( \bar{u}_r-w_r) e_j\dd x\right)^2\dd r\right)^{1/2}.
\] 
It follows immediately that 
 \begin{align*}
    \Xi_{s,t}^\epsilon &\lesssim \left(\sum_k\int_s^t  \int_\mathbb{T}\sigma_{k,\epsilon}^2( \bar{u}^\epsilon_r-w_r) e_j^2 +\sigma_{k}^2( \bar{u}_r-w_r) e_j^2\dd x\dd r\right)^{1/2}\\
      &=\left(\int_s^t  \int_\mathbb{T}\Sigma_{\epsilon}^2( \bar{u}^\epsilon_r-w_r) e_j^2 +\Sigma^2( \bar{u}_r-w_r) e_j^2\dd x\dd r\right)^{1/2}.
 \end{align*}
i.e. $\Xi^\epsilon_{s,t}$ is uniformly bounded in $\epsilon>0$. Concerning the remaining term, we first observe that by similar  arguments as in Lemma \ref{extending the integral} we have 
\[
\sum_k\int_s^t  \left(\int_\mathbb{T}\sigma_{k,\epsilon}( \bar{u}^\epsilon_r-w_r) e_j -\sigma_{k}( \bar{u}^\epsilon_r-w_r) e_j\dd x\right)^2\dd r\to 0.
\]
 It therefore remains to consider 
\begin{align*}
    &\sum_k\int_s^t  \left(\int_\mathbb{T}\sigma_{k}( \bar{u}^\epsilon_r-w_r) e_j -\sigma_{k}( \bar{u}_r-w_r) e_j\dd x\right)^2\dd r\\
   &\lesssim \sum_k\int_s^t \int_\mathbb{T} \left(\sigma_{k}( \bar{u}_r-(\bar{u}_r-\bar{u}^\epsilon_r)-w_r) e_j -\sigma_{k}( \bar{u}_r-w_r) e_j\right)^2\dd x\dd r\\
   &= \sum_k\norm{\sigma_{k}( \bar{u}_r-(\bar{u}_r-\bar{u}^\epsilon_r)-w_r) e_j -\sigma_{k}( \bar{u}_r-w_r) e_j}_{L^2_{t,x}}^2
\end{align*}
Remark that, upon passing to a further subsequence, we may assume $u^\epsilon\to u$ uniformly in $(t, x)\in [0, T]\times \mathbb{T}$. Hence, by continuity of the translation operator in $L^2_{t,x}$, we can conclude that indeed  
\[
\sum_k\int_s^t  \left(\int_\mathbb{T}\sigma_{k}( \bar{u}^\epsilon_r-w_r) e_j -\sigma_{k}( \bar{u}_r-w_r) e_j\dd x\right)^2\dd r\to 0,
\]
yielding the claim.
\end{proof}

By the above Lemma, we may now pass to the limit in \eqref{martingale_epsilon}, obtained for the martingale  $M_t:=\bar{u}_t-u_0-\int_0^t \Delta \bar{u}_r\dd r$.
\begin{cor}
\label{step two conv}
Fix Assumptions \ref{fixing null set}. For $i,j\in \mathbb{N}$, the following processes are martingales with respect to the filtration $(\bar{\mathcal{F}}_t)_t$. 
\begin{equation}
    \begin{split}
          t &\to M^{j}_{t}=\langle
          \left(\bar{u}_t-u_0-\int_0^t \Delta \bar{u}_r\dd r\right), e_j\rangle, \\
  t &\to (M^j_t)^2-\sum_k\int_0^t|\int_\mathbb{T}\sigma_{k}( \bar{u}_r-w_r) e_j\dd x |^2 \dd r,\\
t &\to M^{j}_t  \langle \bar{W}_t, e_i\rangle -(\mathcal{I}a^{i,j})_t.
    \end{split} \label{martingale-limit}
\end{equation}
\end{cor}

In order to conclude that the so obtained martingale $M$ coincides with the stochastic integral 
\[
t\to \int_0^t \sigma(\bar{u}_s-w_s)d\bar{W}_s,
\]
(which is well defined in this setting thanks to Lemma \ref{extending the integral}) and the available regularity of $\bar{u}$ as obtained in \eqref{a priori bounds limit} we need to extend \cite[Proposition A.1]{hofmanova2013degenerate} to our sewing setting. That is precisely the content of the next Lemma \ref{martingale identification}.

\begin{lem}
Fix assumption \ref{fixing null set}. Suppose that for $i,j \in \mathbb{N}$ the processes in \eqref{martingale-limit} are martingales. Then we have 
\[
M_t=\int_0^t \sigma(\bar{u}_s-w_s)d\bar{W}_s.
\]
\label{martingale identification}
\end{lem}

\begin{proof}
We show that for any $j\in \mathbb{N}$ 
\[
\bar{\mathbb{E}}[\langle M_t-\int_0^t \sigma(\bar{u}_r-w_r)d\bar{W}_r, e_j\rangle  ^2]=0.
\]
Let $\sigma_\epsilon$ be again a cut-off mollification. Note that by definition (refer to Lemma \ref{extending the integral}), 
\[
\lim_{\epsilon\to 0}\bar{\mathbb{E}}[\langle \left(\int_0^t \sigma(\bar{u}_s-w_s)d\bar{W}_s-\int_0^t \sigma_\epsilon(\bar{u}_s-w_s)d\bar{W}_s\right), e_j\rangle ^2]=0.
\]
Hence, it suffices to show
\[
\bar{\mathbb{E}}[\langle M_t-\int_0^t \sigma_\epsilon(\bar{u}_s-w_s)d\bar{W}_s, e_j\rangle^2]\to 0.
\]
By similar computations as in earlier proofs, we observe that 
\begin{align*}
&\bar{\mathbb{E}}[\langle M_t-\int_0^t \sigma_\epsilon(\bar{u}_s-w_s)d\bar{W}_s, e_j\rangle ^2]\\
=&\bar{\mathbb{E}}[\langle M_t, e_j\rangle^2]+\bar{\mathbb{E}}[\langle  \int_0^t \sigma_\epsilon(\bar{u}_s-w_s)d\bar{W}_s, e_j\rangle^2]-2\bar{\mathbb{E}}[\langle M_t, e_j\rangle \langle  \int_0^t \sigma_\epsilon(\bar{X}_s-w_s)d\bar{W}_s, e_j\rangle ]\\
=&\mathbb{E}\sum_k\int_0^t|\int_\mathbb{T}\sigma_{k}( \bar{u}_r-w_r) e_j\dd x |^2 \dd r+\bar{\mathbb{E}}[\sum_k\int_0^t|\int_\mathbb{T}\sigma_{k, \epsilon}( \bar{u}_r-w_r) e_j\dd x |^2 \dd r]\\
&-2\bar{\mathbb{E}}[\langle M_t, e_j\rangle \langle  \int_0^t \sigma_\epsilon(\bar{X}_s-w_s)d\bar{W}_s, e_j\rangle ].
\end{align*}
Concerning the third term after the last equality, we need to show that 
\begin{equation}
    \bar{\mathbb{E}}[\langle M_t, e_j\rangle \langle  \int_0^t \sigma_\epsilon(\bar{u}_s-w_s)d\bar{W}_s, e_j\rangle ]=\bar{\mathbb{E}}[\sum_k\int_0^t \left(\int_\mathbb{T}\sigma_{k}( \bar{u}_r-w_r) e_j\dd x\right)\left(\int_\mathbb{T}\sigma_{k, \epsilon}( \bar{u}_r-w_r) e_j\dd x \right)].
    \label{g sewing}
\end{equation}
To this end, note that the process 
$
t\to \sigma_{k, \epsilon}(\bar{u}_t-w_t)\in L^2_x
$
is progressively measurable and in $L^2(\Omega\times [0,T]\times \mathbb{T})$. Hence we can approximate it by elementary processes, i.e. take 
\[
\sigma_{\epsilon, N}(s)e_k:=\sum_{i=1}^N \sigma_{k, \epsilon}(\bar{u}_{t_i}-w_{t_i})1_{[t_i, t_{i+1})}(s),
\]
 where $s=t_1<t_2<\dots <t_{N+1}=t$. Then following some algebraic manipulations, relating the integral to standard sewings as earlier, we find 
% \rmm{to get the first equality, we write $M_{t}-M_{s}=\sum_{i}(M_{t_{i+1}}-M_{t_{i}})$ so we have product of two sums. then we want to say }
\begin{align*}
    &\bar{\mathbb{E}}[\langle M_t-M_s, e_j\rangle \langle  \int_s^t \sigma_{\epsilon, N}(r)d\bar{W}_r, e_j\rangle  |\bar{\mathcal{F}}_s]\\
    =& \sum_i^N \sum_{k=1}\bar{\mathbb{E}}[\langle \sigma_{k, \epsilon}(\bar{u}_{t_i}-w_{t_i}), e_j\rangle \bar{\mathbb{E}}[\langle M_{t_{i+1}}, e_j\rangle \langle \bar{W}_{t_{i+1}}, e_k\rangle-\langle M_{t_{i}}, e_j\rangle \langle \bar{W}_{t_{i}}, e_k\rangle|\bar{\mathcal{F}}_{t_i}]|\bar{\mathcal{F}}_s] \\
    =& \sum_i^N\sum_{k=1}\bar{\mathbb{E}}[\langle \sigma_{k, \epsilon}(\bar{u}_{t_i}-w_{t_i}), e_j\rangle \bar{\mathbb{E}}[(\mathcal{I}a^{k,j})_{t_i, t_{i+1}}|\bar{\mathcal{F}}_{t_i}]\bar{\mathcal{F}}_{s}]\\
    =& \bar{\mathbb{E}}[\sum_i^N\sum_{k=1}\langle \sigma_{k, \epsilon}(\bar{u}_{t_i}-w_{t_i}), e_j\rangle(\mathcal{I}a^{k,j})_{t_i, t_{i+1}})|\bar{\mathcal{F}}_s].
\end{align*}

Upon taking expectation we obtain the identity 
\begin{equation}
    \bar{\mathbb{E}}[\langle M_t-M_s, e_j\rangle \langle  \int_s^t \sigma_{\epsilon, N}(r)d\bar{B}_r, e_j\rangle]={\mathbb{E}}\sum_i^N\bar[\sum_{k=1}\langle \sigma_{k, \epsilon}(\bar{u}_{t_i}-w_{t_i}), e_j\rangle(\mathcal{I}a^{k,j})_{t_i, t_{i+1}}].
    \label{riemann-stieltjes}
\end{equation}
We will now show that the above converges as a Riemann-Stieltjes integral. With this aim in mind, we begin to  observe  that $t\to \langle \sigma_{k, \epsilon}(\bar{u}_{t}-w_{t}), e_j\rangle$ is continuous and bounded. Moreover, the function $t\to (\mathcal{I}a^{k,j})_{t}$ is of bounded variation $\bar{\mathbb{P}}$-almost surely. Indeed, recall that $(\mathcal{I}a^{k,j})_{t}$ denotes the sewing of the germ $a^{i, j}_{s,t}=\langle (\sigma_i*L_{s,t})(\bar{u}_s), e_j\rangle$.  One can then  readily verify that for any partition $P=\{s_1\leq \ldots \leq n_P\}$ of $[0,T]$ we have 
\begin{equation}
    \sum_{s_i\in P}|(\mathcal{I}a^{k,j})_{s_{i+1}}-(\mathcal{I}a^{k,j})_{s_{i}}|\leq \left((\mathcal{I}b^{k})_{T}\right)^{1/2},
    \label{BV trick}
\end{equation}
where $b^{k}=\int_\mathbb{T} (\sigma_k^2*L_{s,t})(\bar{u}_s)$. In the case of smooth $\sigma_k$ this follows essentially from the following estimates
\begin{align*}
    \sum_{s_i\in P}|(\mathcal{I}a^{k,j})_{s_{i+1}}-(\mathcal{I}a^{k,j})_{s_{i}}|&=\sum_{s_i\in P}|\int_{s_i}^{s_{i+1}}\langle (\sigma_k(\bar{u}_r-w_r), e_j\rangle \dd r|\\
    &\lesssim \sum_{s_i\in P}\int_{s_i}^{s_{i+1}}|\langle (\sigma_k(\bar{u}_r-w_r), e_j\rangle|\dd r\\
    &\lesssim \left(\int_0^T\int_{\mathbb{T}}\sigma_k^2(\bar{u}_r-w_r)\dd x\dd r\right)^{1/2}.
\end{align*}
The general case $\sigma_k^2\in L^p$ being again a consequence of a mollification argument and the stability of the sewings (see Lemma \ref{identification}). Using the notation $V_{[0, T]}(f)$ for the total variation norm of  a function $f$ along the interval $[0, T]$, we therefore conclude from \eqref{BV trick} that 
\[
\bar{\mathbb{E}}[(V_{[0,T]}(\mathcal{I}a^{k, j}))^2]\lesssim \bar{\mathbb{E}}[(\mathcal{I}b^{k,j})_{T}]<\infty
\]
meaning that indeed, $t\to (\mathcal{I}a^{k, j})_t$ is of bounded variation $\bar{\mathbb{P}}$-almost surely. The right hand side of \eqref{riemann-stieltjes} is therefore  well defined as a Riemann-Stieltjes integral, for which we have the bound
\begin{align*}
   \left| \bar{\mathbb{E}}[\sum_i^N\sum_{k=1}\langle \sigma_{k, \epsilon}(\bar{u}_{t_i}-w_{t_i}), e_j\rangle(\mathcal{I}a^{k,j})_{t_i, t_{i+1}})]\right|&\lesssim \sum_{k=1} \bar{\mathbb{E}}[\norm{\sigma_{k, \epsilon}}_\infty V_{[0,T]}(\mathcal{I}a^{k, j})]<\infty
\end{align*}
 Hence, denoting $g^{k,j}_t:=\langle  \sigma_{k, \epsilon}(\bar{u}_{t}-w_{t}) e_j\rangle$ and $h^{j,k}_t:=(\mathcal{I}a^{k,j})_t$ we have
\[
\lim_{N\to \infty}\sum_i^N \bar{\mathbb{E}}[\sum_{k=1}\langle \sigma_{k, \epsilon}(\bar{u}_{t_i}-w_{t_i}), e_j\rangle(\mathcal{I}a^{k,j})_{t_i, t_{i+1}})]=\bar{\mathbb{E}}[\sum_{k=1} \int_0^t g^{k,j}_s dh^{j,k}_s],
\]
where the right hand side integrals are understood as Riemann-Stieltjes integrals. Finally, again by a mollification argument on $\sigma$  similar to Lemma \ref{identification} one verifies that 
\[
\sum_{k=1} \int_0^t g^{k,j}_s dh^{j,k}_s=\bar{\mathbb{E}}[\sum_k\int_0^t \left(\int_\mathbb{T}\sigma_{k}( \bar{u}_r-w_r) e_j\dd x\right)\left(\int_\mathbb{T}\sigma_{k, \epsilon}( \bar{u}_r-w_r) e_j\dd x \right)]
\]
meaning we have established \eqref{g sewing}.

Overall we therefore conclude that 
\begin{align*}
    &\bar{\mathbb{E}}[\langle M_t-\int_0^t \sigma_\epsilon(\bar{u}_s-w_s)d\bar{W}_s, e_j\rangle ^2]\\
    =&\mathbb{E}\sum_k\int_0^t|\int_\mathbb{T}\sigma_{k}( \bar{u}_r-w_r) e_j\dd x |^2 \dd r+\bar{\mathbb{E}}[\sum_k\int_0^t|\int_\mathbb{T}\sigma_{k, \epsilon}( \bar{u}_r-w_r) e_j\dd x |^2 \dd r]\\
&-2\bar{\mathbb{E}}[\sum_k\int_0^t \left(\int_\mathbb{T}\sigma_{k}( \bar{u}_r-w_r) e_j\dd x\right)\left(\int_\mathbb{T}\sigma_{k, \epsilon}( \bar{u}_r-w_r) e_j\dd x \right)].
\end{align*}
Due to the stability of the corresponding sewings (Lemma~\ref{regularity of averaging operator} and Lemma \ref{sewing convergence}) used similarly as in Lemma \ref{main convergence lemma}, we may indeed conclude our claim that 
\[
M_t=\int_0^t \sigma_\epsilon(\bar{u}_s-w_s)d\bar{W}_s.
\]

\end{proof}
In summary, this concludes the proof of Theorem \ref{main theorem}.

\appendix
\section{Appendix}
\subsection{Local time and occupation times formula}\label{app: local time section}
We recall for the reader the basic concepts of occupation measures, local times and the occupation times formula. For a comprehensive review paper on these topics, see  \cite{horowitz}. 
\begin{defn}
Let $w:[0,T]\to \RR^d$ be a measurable path. Then the occupation measure at time $t\in [0,T]$, written $\mu^w_t$, is the Borel measure on $\RR^d$ defined by 
\[
\mu^w_t(A):=\lambda(\{ s\in [0,t]:\ w_s\in A\}), \quad A\in \mathcal{B}(\RR^d),
\]
where $\lambda$ denotes the standard Lebesgue measure. 
\end{defn}
The occupation measure thus measures how much time the process $w$ spends in certain Borel sets. Provided for any $t\in [0,T]$, the measure is absolutely continuous with respect to the Lebesgue measure on $\RR^d$, we call the corresponding Radon-Nikodym derivative local time of the process $w$:
\begin{defn}
Let $w:[0,T]\to \RR^d$ be a measurable path. Assume that there exists a measurable function $L^w:[0,T]\times \RR^d\to \RR_+$ such that 
\[
\mu^w_t(A)=\int_A L^w_t(z)dz, 
\]
for any $A\in \mathcal{B}(\RR^d)$ and  $t\in [0,T]$. Then we call $L^w$ local time of $w$. 
\end{defn}
Note that by the definition of the occupation  measure, we have for any bounded measurable function $f:\RR^d\to \RR$ that 
\begin{equation}
    \int_0^tf(w_s)\dd s=\int_{\RR^d} f(z)\mu^w_t(dz).
    \label{occupation times formula}
\end{equation}
The above equation \eqref{occupation times formula} is called occupation times formula. Remark that in particular, provided $w$ admits a local time, we also have for any $x\in \RR^d$
\begin{equation}
    \int_0^tf(x-w_s)\dd s=\int_{\RR^d} f(x-z)\mu^w_t(dz)=\int_{\RR^d}f(x-z)L^w_t(z)dz=(f*L^w_t)(x).
\end{equation}
\subsection{The Sewing Lemma}
We recall the Sewing Lemma due to \cite{gubi} (see also \cite[Lemma 4.2]{frizhairer}). Let $E$ be a Banach space, $[0,T]$ a given interval. Let $\Delta_n$ denote the $n$-th simplex of $[0,T]$, i.e. $\Delta_n:\{(t_1, \dots, t_n)| 0\leq t_1\dots\leq t_n\leq T \} $. For a function $A:\Delta_2\to E$ define the mapping $\delta A: \Delta_3\to E$ via
\[
(\delta A)_{s,u,t}:=A_{s,t}-A_{s,u}-A_{u,t}
\]
Provided $A_{t,t}=0$ we say that for $\alpha, \beta>0$ we have $A\in C^{\alpha, \beta}_2(E)$ if $\norm{A}_{\alpha, \beta}<\infty$ where
\[
\norm{A}_\alpha:=\sup_{(s,t)\in \Delta_2}\frac{\norm{A_{s,t}}_E}{|t-s|^\alpha}, \qquad \norm{\delta A}_{\beta}:=\sup_{(s,u,t)\in \Delta_3}\frac{\norm{(\delta A)_{s,u,t}}_E}{|t-s|^\beta} \qquad \norm{A}_{\alpha, \beta}:=\norm{A}_\alpha+\norm{\delta A}_\beta
\]
For a function $f:[0,T]\to E$, we note $f_{s,t}:=f_t-f_s$

Moreover, if for any sequence $(\mathcal{P}^n([s,t]))_n$ of partitions of $[s,t]$ whose mesh size goes to zero, the quantity 
\[
\lim_{n\to \infty}\sum_{[u,v]\in \mathcal{P}^n([s,t])}A_{u,v}
\]
converges to the same limit, we note
\[
(\mathcal{I} A)_{s,t}:=\lim_{n\to \infty}\sum_{[u,v]\in \mathcal{P}^n([s,t])}A_{u,v}.
\]

\begin{lem}[Sewing]
Let $0<\alpha\leq 1<\beta$. Then for any $A\in C^{\alpha, \beta}_2(E)$, $(\mathcal{I} A)$ is well defined. Moreover, denoting $(\mathcal{I} A)_t:=(\mathcal{I} A)_{0,t}$, we have $(\mathcal{I} A)\in C^\alpha([0,T], E)$ and $(\mathcal{I} A)_0=0$ and for some constant $c>0$ depending only on $\beta$ we have
\[
\norm{(\mathcal{I} A)_{t}-(\mathcal{I} A)_{s}-A_{s,t}}_{E}\leq c\norm{\delta A}_\beta |t-s|^\beta.
\]
We say the germ $A$ admits a sewing $(\mathcal{I}A)$ and call $\mathcal{I}$ the sewing operator. 
\label{sewing}
\end{lem}
Let us finally cite a result allowing to commute limits and sewings. 
\begin{lem}[Lemma A.2 \cite{Galeati2021}]
\label{sewing convergence}
For $0<\alpha\leq 1<\beta$ and $E$ a Banach space, let  $A\in C^{\alpha, \beta}_2(E)$ and $ (A^n)_n\subset C^{\alpha, \beta}_2(E)$ such that for some $R>0$, $\sup_{n\in \mathbb{N}}\norm{\delta A^n}_\beta\leq R$ and such that $\norm{A^n-A}_\alpha\to 0$. Then \[\norm{\mathcal{I}(A-A^n)}_\alpha\to 0.\]
\end{lem}

\subsection{Auxiliary Lemmata}

    \begin{lem}[A Schauder estimate]
    \label{schauder}
For any  $s\leq t$ and $\theta\in [0,1]$ and $\rho/2+\theta\geq 0$ we have that 
\begin{equation}
    \norm{(-\Delta)^{\rho/2}(P_t-P_s)}_{L(H, H)}\lesssim (t-s)^\theta s^{-(\theta+\rho/2)}
\end{equation}
\end{lem}
\begin{proof}
We have for $k\in \mathbb{Z}$
\begin{equation*}
(-\Delta)^{\rho/2}(P_t-P_s)e_k
    =k^{\rho}(e^{-k^2t}-e^{-k^2s})e_k.
\end{equation*}
We note that 
\[
|e^{-k^2t}-e^{-k^2s}|=|\int_{k^2s}^{k^2t}-e^{-x}\dd x|\leq k^2(t-s), \qquad |e^{-k^2t}-e^{-k^2s}|\leq 2e^{-k^2s} 
\]
Interporlating the above bounds yields
\begin{equation*}
     k^{\rho}(e^{-k^2t}-e^{-k^2s})\lesssim (t-s)^\theta (k^2s)^{\rho/2+\theta}e^{-k^2s(1-\theta)}s^{-(\rho/2+\theta)}\lesssim (t-s)^\theta s^{-(\rho/2+\theta)}
\end{equation*}
Thus, for any $v\in H$  we have 
\[
\norm{(-\Delta)^{\rho/2}(P_t-P_s)v}_{H}=(t-s)^\theta s^{-(\rho/2+\theta)}\left(\sum_k \langle v, e_k\rangle^2 \right)^{1/2}=(t-s)^\theta s^{-(\rho/2+\theta)}\norm{v}_H
\]
yielding the claim.
\end{proof}

\bibliography{biblio}
\bibliographystyle{alpha}
\end{document}